\documentclass[11pt]{article}

\usepackage{amstext, amsmath,latexsym,amsbsy,amssymb}
\usepackage{enumitem}

\newtheorem{proposition}{Proposition}[section]
\newtheorem{remark}{Remark}[section]
\newenvironment{proof}{{\bf Proof\ }}{\QED\\}
\newtheorem{lemma}{Lemma}[section]

\numberwithin{equation}{section}
\newtheorem{theorem}{Theorem}[section]
\newtheorem{corollary}{Corollary}[section]

\newcommand{\QED}{\hspace*{\fill}\rule{2.5mm}{2.5mm}}

\usepackage{amssymb}\usepackage{graphicx}
\usepackage{tikz}

\usepackage{color}
\newcommand\qed{\hfill$\sqcap\kern-7.5pt\hbox{$\sqcup$}$}

\newcommand{\NN}{\mathbb{N}}
\newcommand{\RR}{\mathbb{R}}

\newcommand{\Sx}{\mathbb{S}}

\newtheorem{theo}{Theorem}
\newtheorem{prop}[theo]{Proposition}

\newcommand{\beqn}{\begin{equation}}
\newcommand{\eeqn}{\end{equation}}
\newcommand{\bear}{\begin{eqnarray}}
\newcommand{\eear}{\end{eqnarray}}
\newcommand{\bean}{\begin{eqnarray*}}
\newcommand{\eean}{\end{eqnarray*}}

\allowdisplaybreaks
\begin{document}
\title{Convergence to equilibrium of a linearized quantum Boltzmann equation for bosons at very low temperature}
\author{\\ Miguel Escobedo'* and Minh-Binh Tran*\\\\
'Departamento de Matem\'aticas, Universidad del Pa\'is Vasco\\ (UPV/EHU) Apartado 644, E48080 Bilbao, Spain\\
* Basque Center for Applied Mathematics\\ Mazarredo 14, 48009 Bilbao Spain\\
Email:  miguel.escobedo@ehu.es, tbinh@bcamath.org}

\maketitle
\begin{abstract}
We consider an approximation of the linearised equation of the homogeneous Boltzmann equation that describes  the distribution of quasiparticles in a dilute gas of bosons at low temperature. The corresponding collision frequency is neither bounded from below nor from above. We prove the existence and uniqueness of solutions  satisfying the conservation of energy. We show that these solutions converge to the corresponding stationary state, at an algebraic rate as time tends to infinity.

\end{abstract}
{\bf Keyword:}
{quantum Boltzmann equation, rate of convergence to equilibrium, algebraic decay. 
\section{Introduction}

\label{LQBE}
A kinetic equation that describes the evolution of a non equilibrium spatially homogeneous distribution $n(t, p)$ of quasiparticles in a dilute Bose gas below the Bose Einstein transition temperature $T_c$  has been obtained by several authors (see for example \cite{IG}, \cite{KD1}, \cite{KD2}) and reads as follows:
\bear
&&\frac {\partial n} {\partial t}(t, p)=\int _{ \RR^3}\int _{ \RR^3} dp_1dp_2 \left[R(p, p_1, p_2)-R(p_1, p, p_2)-R(p_2, p_1, p) \right]
\label{E1}\\
&&R(p, p_1, p_2)=|\mathcal M(p, p_1, p_2)|^2\left[\delta (\omega (p)-\omega (p_1)-\omega (p_2))  \delta (p-p_1-p_2)\right]\times\nonumber \\
&&\hskip 2cm \times\left[ n(p_1)n(p_2)(1+n(p))-(1+n(p_1)(1+n(p_2))n(p) \right]\label{E2}
\eear
where $\mathcal M(p, p_1, p_2)$ is the transition probability, $\omega (p)$ is the so called Bogoliubov dispersion law:
\bear
\omega (p)=\left[\frac {g n_c} {m}|p|^2+\left(\frac {|p|^2} {2m} \right)^2 \right]^{1/2} \label{E3}
\eear
$m$ is the mass of the particles, $g$ is the interaction coupling constant and $n_c$ is the density of particles in the superfluid.  It is well known that the equation (\ref{E1})--(\ref{E3}) has a family of equilibria:
\bear
\label{E3bis}
n_0(p )=\frac {1} {e^{\frac{\omega (p)}{k_B T}}-1},\,\,\,\beta >0.
\eear
where $k_B$ is the Boltzmann's constant and $T$ the temperature of the quasiparticles whose distribution is $n_0$.

The relaxation of  $n$ towards its corresponding  equilibrium is a question that has deserved some interest by several authors (cf.  \cite{Buot:ORT:1997}, \cite{E}, \cite{EPV}, \cite{IG}). In the  more strictly mathematical literature,  the convergence to equilibrium of Boltzmann equation has been extensively studied and still is. Since the works by  T. Carleman \cite{MR1555365} and  H. Grad \cite{Grad:1963:ATB2},  then by  L. Arkeryd \cite{ark},  S. Ukai and K. Asano \cite{Ukai},  G. Toscani \cite{Toscani} and L. Desvillettes \cite{Desv90} until those by L. Desvillettes and C. Villani \cite{DesvillettesVillani:2005:OTG} and later by Y. Guo and R. Strain \cite{GS} (cf. the review article \cite{MR1942465} for more detailed references).  However, we do not consider in this work the nonlinear problem (\ref{E1})--(\ref{E3}). We only study, instead, the relaxation process of the equation linearised around one equilibrium.  Let us then write:
\bear
n(t, p)&=&n_0(p)+n_0(p)[1+n_0(p)]\Omega (t, p) \label{E4}\\
&=&n_0(p)+\frac {\Omega (t, p)} {4 \, \sinh^2 \left(\frac{\omega (p)}{2k_BT}\right)}\label{E4ii}
\eear
Plugging this expression in the equation and keeping only the linear terms in $\Omega $ we obtain:
\bear
n_0(p)[1+n_0(p )]\frac {\partial \Omega } {\partial t}(t, p)&=&\mathcal L(\Omega )(t, p) \label{E5i}\\
\mathcal L(\Omega )(t, p)&=&-M(p)\,\Omega  (t, p)+\mathcal T (\Omega )(t, p)\label{E5ii}\\
\mathcal T (\Omega )(t, p)&=&\int  _{ \RR^3 }\mathcal U(p, p') \Omega (t, p') dp'\label{E5iii}
\eear
where the measure $\mathcal U(p, p')$ and the function $M(p)$ have been calculated  in \cite{EPV} and whose explicit expressions are recalled in formulas (\ref{S6EU}) and (\ref{S6EM}) of the  Appendix. 

The structure of the equation (\ref{E5i})--(\ref{E5iii}) is the same as in other linearised Boltzmann equations, as they may be seen for example in  \cite{Caflisch:1980:BES}, \cite{CIP},  \cite{EPV}, \cite{Grad:1963:ATB2}, \cite{Spohn}. 

The relaxation to equilibrium  of the solutions of (\ref{E8bis})--(\ref{E12b}) has been considered in  \cite{Buot:ORT:1997},  \cite{Claro}, \cite{E}, \cite{EPV}, \cite{IG}. 

 As it is well known, the properties of the operator $\mathcal L$ crucially depend  on the range of the function $M (p)$ and compactness properties of the integral operator $\mathcal T$. For the classical Boltzmann equation with hard potential the corresponding function $M$ is such that, for some constant $M _0>0$, $M (p)\to M_0$ as $|p|\to 0$, $M (p)\to +\infty$ as $|p|\to \infty$, and its range is  $[M _0, +\infty)$. For soft potentials, $M (p)\to M_0>0$ as $|p|\to 0$ but $M (p)\to 0$ as $|p|\to \infty$ and the range is $[0, M _0]$. In both cases the integral operator $\mathcal T$ is compact in some suitable functional space.  It was shown in \cite{Buot:ORT:1997} that the values of the function $M(p)$  in (\ref{E5ii}) range from zero to $\infty$ as the variable $|p|$ goes from zero to $\infty$ (see Lemma \ref{Gamma} in the Appendix below). From this point of view, the situation for equation (\ref{E5i})--(\ref{E5iii}) is then similar to the case of the soft potentials for cl
 assical 
 particles.

In the case of the spatially homogeneous linearized Boltzman equation for classical particles with soft potential it was observed in \cite{Grad:1963:ATB2} (see also \cite{Caflisch:1980:BES} and \cite{Ukai}) that the  spectrum of  the corresponding linearised  operator $\mathcal L$ goes down until the origin   and no exponential  rate of convergence can  be expected for the solutions.  It is shown in  \cite{Caflisch:1980:BES} that for soft potentials and spatially homogeneous initial data $f(0, p)$ decaying exponentially fast as $|p|\to \infty$, the part of the solution $f$ in the range of $\mathcal L$ decays  in $L^2(\RR^3)$  like $e^{-\lambda t^\theta}$ for some $\lambda >0$ and $\theta \in (0, 1)$. On the other hand,  for non homogeneous initial data, the authors of  \cite{Ukai} proved algebraic rates of decay in Lebesgue--Sobolev mixed type spaces.

\subsection{Approximation of the linearised equation.}

Since  the functions $\omega (p)$ and $\mathcal M(p, p_1, p_2)$ appearing in equation  (\ref{E5i})--(\ref{E5iii}) are complicated functions of their arguments, we restrict the range of our analysis. Following \cite{Buot:ORT:1997} we consider the situation where the equilibria $n_0$ in (\ref{E5i})--(\ref{E5iii}) is at a quite  low temperature $T$. More precisely, we suppose  that the temperature $T$, the density $n_c$ of superfluid  and the interaction coupling constant $g$ are  such that $k_BT$ is much smaller than $gn_c$.
That range  has been widely considered in the physics literature, where the functions $\omega (p)$ and $\mathcal M(p, p_1, p_2)$ are then approximated as follows:
\bear
&&\omega (p)= c|p|, \,\,\,c=\sqrt{\frac {gn_c} {m}} \label{E6}\\
&&|\mathcal M(p, p_1, p_2)|^2=\frac {9c} {64\pi ^2mn_c}|p||p_1||p_2|, \label{E7}
\eear
 (cf. \cite{Buot:ORT:1997}, \cite{E}, \cite{IG}, \cite{Benin}). This approximation has an important consequence. Indeed, if $\omega (p)=c|p|$, then the  condition
$\omega (p)=\omega (p')+\omega (p-p')$ reads $|p|=|p'|+|p-p'|$. This implies that $p$ and $p'$ must be  parallel vectors of $\RR^3$.  The domain of integration in the integral at the right hand side of equation (\ref{E5i})--(\ref{E5iii}) is then reduced to the set $\mathcal C_p= \{\lambda\, p;\,\, \lambda \in \RR \}
$. More precisely,  we are approximating the equation (\ref{E5i})--(\ref{E5iii}) by
\bear
n_0(p )[1+n_0(p )]\frac {\partial \Omega } {\partial t}(t, p)=-M(p)\Omega  (t, p)+\int  _{ \RR^3 }
 \Omega (t, p')\,W(p, p')dp'  \label{E51}
 \eear 
where $W(p, p')$ and $M(p)$ are defined by (\ref{S6EW}) and (\ref{S6EG}) in the Appendix. Our goal  is to study the solutions of the Cauchy problem  associated to equation (\ref{E51}), their existence, uniqueness  and  relaxation towards equilibrium. 

Due to the formulas (\ref{E4}), (\ref{E4ii}), and for the sake of notation we shall use the following convention all along this article. Given $p\in \RR^3$, we shall denote:
\bear
\label{E8-1}
k\equiv k(p)=\frac {c |p|} {2k_BT}.
\eear
Since, we will also denote $|p|=r$, we shall use sometimes 
\bear
\label{E8-3}
k=\frac {c r} {2k_BT}.
\eear
With some abuse of notations we will also write $n_0(p)=n_0(|p|)=n_0(r)$ and also, by (\ref{SAEM0}), $M(p)=M(|p|)=M(r)$.
\begin{proposition} 
\label{Prop1}Let $\left\{ Y _{\ell\, m  }\right\} _{ \ell, m }$ be the spherical harmonics on $\Sx^2$. For any  sequence $\{ c _{ \ell\, m }\}$ of real numbers such that:
$$ \sum _{ \ell=0 }^\infty\sum _{ n=-\ell }^\ell c^2 _{ \ell\, m }<\infty$$
define 
$$
\Theta (p)=\left(\sum _{ \ell=0 }^\infty \sum _{ m=-\ell }^\ell c_{ \ell\, m }Y _{ \ell\, m }\left(\frac {p} {|p|} \right)\right) |p|.
$$
Then:
\bean
&&(i) \quad  \Theta \in   L^2\left(\RR^3, \frac {dp} {\sinh ^{2}k} \right),\\
&&(ii)\quad  -M(p)\Theta (p)+\int  _{ \RR^3 } \Theta (p')\,W(p, p')dp'=0.
\eean
\end{proposition}
\begin{theorem}\label{theorem}
Suppose that  $\Omega _0\in L^2\left(\RR^3, \frac {dp} {\sinh ^{2}k} \right)$.
Then, there exists a unique function $\Omega (t, p)$ such that
 \bear
 &&\Omega \in L^\infty \left(0,\infty;  L^2\left(\RR^3, \frac {dp} {\sinh ^{2}k} \right)\right)\cap 
 C \left([0,\infty); L^2\left(\RR^3, \frac {dp} {\sinh ^{2}k} \right)\right),\nonumber \\\label{S1E250}\\
&& \Omega - \Theta \in L^2\left(0, \infty; L^2\left(\RR^3, M(p)dp \right)\right),\label{S1E251}\\
&& \frac {\partial \Omega } {\partial t}\in L^2\left(0, \infty;  L^2\left(\RR^3, \frac {dp} {M(p)\, \sinh^4k}\right)\right), \label{S1E252}
 \eear
 satisfying the equation (\ref{E51}) in $L^2\left(0, \infty;  L^2\left(\RR^3, \frac {dp} {M(p)\, \sinh^4 k}\right)\right)$ and taking the initial data $\Omega _0$ in the following sense:
\begin{equation}
\label{S1Einitial}
\lim _{ t\to 0 }\left(||\Omega (t)-\Omega _0|| _{ L^2\left(\RR^3, \frac {dp} {M(p)\, \sinh^4k}\right)}+
||\Omega (t)-\Omega _0|| _{ L^2\left(\RR^3, \frac {dp} {\sinh ^{2}k} \right)}\right)=0.
\end{equation}
This solution also satisfies the following conservation property:
\bear
\label{consenergyM}
\frac {d} {dt}\int  _{ \RR^3 }n_0(p)(1+n_0(p))\Omega (t, p)|p|dp=0.
\eear
If $\Omega _0$ satisfies also:
\bear
\label{S1Econdition}
\int  _{ |p|<1 }\frac {|\Omega_0 (p)|^2} {|p| \sinh ^2k}dp<\infty
\eear
then
\bear
\label{S1ERatedecay}
||\Omega (t)-\Theta ||_{L^2\left(\RR^3, \frac {dp } {\sinh^{2} k}\right)  }
\le  \frac{C}{(1+t)^{1/2 }}||\Omega _0-\Theta||_{L^2\left(\RR^3, \frac {dp } {\sinh^{2} k} \right)  },
\eear
where
\bear
&&\hskip -0.5cm\Theta (p)=\left(\sum _{ \ell=0 }^\infty \sum _{ m=-\ell }^\ell c_{ \ell\, m }Y _{ \ell\, m }\left(\frac {p} {|p|} \right)\right) |p| \label{S1ETeta}\\
&&\hskip -0.5cmc _{ \ell\, m }=\left(\frac {\pi c} {2\sqrt{15}k_BT} \right)^4\int  _{ \RR^3 }\Omega _0(p)n_0(p)(1+n_0(p))Y _{ \ell\, m }\left(\frac {p} {|p|} \right)dp.\label{S1Eclm}
\eear
\end{theorem}
\begin{remark}
\label{convergencerate}
The algebraic decay rate  in  (\ref{S1ERatedecay}) is not sufficient to have the integrability in time of $||\Omega (t)-\Theta ||^2_{L^2\left(\RR^3, \frac {dp } {\sinh^{2} k}\right)  }$ at infinity although, by (\ref{S1E251}), this integrability property is true for $||\Omega (t)-\Theta ||^2_{L^2\left(\RR^3,\, M(p) dp \right)  }$.
\end{remark}
\begin{remark}
\label{remarkeme}
The behaviors of the function $M(p)$ as $|p|\to 0$ and $|p|\to \infty$  are given in Proposition \ref{S6P1}  of the Appendix. 
\end{remark}
\begin{remark}
\label{remarkenergy}
The system of quasiparticles described by (\ref{E1})--(\ref{E2}) satisfies the physical property of energy conservation. That property  is expressed, in terms of the function $n(t, p)$ as:
$$
\frac {d} {dt}\int  _{ \RR^3 }n(t, p)\omega (p)dp=0.
$$
 The identity (\ref{consenergyM}) shows that this conservation of energy still holds for the equation (\ref{E51}).
 
 Another natural quantity for the set of  quasiparticles described  by (\ref{E1})-(\ref{E2}) is
$
N(t)=\int _{ \RR^3 }n(t, p)dp
$
that represents the total number of particles. That physical quantity is not conserved by the system of particles described by (\ref{E1})--(\ref{E2}), and the function $N(t)$ is not  preserved, even formally, by equation (\ref{E1})-(\ref{E2}). Nevertheless, the corresponding  quantity for the linearised equation, namely $M(t)=\int _{ \RR^3 }n_0(p)(1+n_0(p))\Omega (t, p)dp$ is well defined for the solutions obtained in Theorem \ref{theorem}. See also   Remark \ref{totalmass} below.

 \end{remark}
 
The proof of theorem (\ref{theorem}) is based on the following argument.
Decompose first $\Omega (t, p)$ in spherical harmonics:
\bear
\Omega (t, p)=\sum _{ \ell=0 }^\infty \sum _{m=-\ell }^\ell\Omega _{ \ell\, m } (t, |p|) Y _{\ell\, m  }\left( \frac {p} {|p|}\right).\label{S1E200}
\eear
Using the decomposition of the measure $W$ in Legendre's polynomial (recalled in the Appendix) we obtain for each $\ell$ and $m$:
\bear
n_0[1+n_0]\frac {\partial \Omega  _{ \ell\, m }} {\partial t}(t, r)=-M (r)\Omega _{ \ell\, m }  (t, r)+\nonumber \\
+\frac {1} {2\ell +1}\int  _{0}^\infty
W  _{ \ell } (r, r') \Omega _{ \ell\, m } (t, r') dr' \label{S1E210}
\eear
where $r=|p|$, $r'=|p'|$ and:
\bear
\frac {1} {2\ell +1}W_\ell (r, r')=\frac {1} {2}\int  _{ -1 }^1W(p, p') P _{ \ell} (u)du,\,\,\,\ell=0, 1, \cdots \label{S1E211}
\eear
It follows from the expression of  $G(k, k')$ and $W(p, p')$  in (\ref{S6EG}) and (\ref{S6EW}) that
\bean
\frac {1} {2\ell +1}W_\ell (r, r')&=&\frac {1} {2}\int  _{ -1 }^1W(p, p') P _{ \ell} (u)du\\
&=&\frac {1} {2}\int  _{ -1 }^1W(p, p') du=G(r, r'),\,\,\,\ell=1, 2 \cdots
\eean
and all the coefficients $W_\ell (r, r')$ are equal. Therefore  all the modes $\Omega  _{\ell\, m }(t, r)$  satisfy the same equation:
\bear
&&\hskip -0.5cm n_0(r)[1+n_0(r)]\frac {\partial \Omega  _{ \ell\, m }} {\partial t}(t, r)=L(\Omega  _{ \ell\, m })(t, r)  \label{E5000-1}\\
&&\hskip -0.5cm L(\Omega  _{ \ell\, m })(t, r)=- M(r)\Omega _{ \ell\, m }  (t, r)
+ \int  _{0}^\infty
W_0 (r, r') \Omega _{ \ell\,m } (t, r') dr' \label{E5000-2}
\eear
for all $\ell=0, 1, 2, \cdots$ and $m\in \left\{-\ell, \cdots,\ell \right\}$, where with some abuse of notation we denote:
$$
n_0(p)=n_0(r)
$$
and $W_0(r, r')$ is given by formula (\ref{SAEW0})   in the Appendix.

Let us then consider an initial data $\Omega _0\in L^2(\RR^3)$ and write its decomposition in spherical harmonics:
\bean
\Omega _0(p)=\sum _{ \ell=0 }^\infty \sum _{ m=-\ell }^\ell\Omega  _{ 0, \ell\, m }(|p|)Y _{ \ell\, m }\left(\frac {p} {|p|} \right).
\eean

The solution to the equation (\ref{E51}) with the initial condition $\Omega (0, p)=\Omega _0(p)$  is then given by the function defined by the series  (\ref{S1E200})  
where every function $\Omega  _{\ell\, m }(t, r)$ solves  the equation (\ref{E5000-1}), (\ref{E5000-2}), with initial data $\Omega  _{ 0, \ell\, m }$, for  $\ell=0, 1, 2, \cdots$ and $m\in \left\{-\ell, \cdots,\ell \right\}$. 

It is then enough to study the solutions of the Cauchy problem for the equation (\ref{E5000-1}), (\ref{E5000-2}).
To this end we perform the following change of variables:
\bear
f(t, k)=\frac{c|p|}{2k_BT} \frac{\Omega (t, p)}{\sinh \left(\frac{c|p|}{2k_BT}\right)}, \,\,\,k=\frac{c|p|}{2k_BT}  \label{E8}
\eear
and obtain  (cf. \cite{Buot:ORT:1997} and \cite{Wannier}):
\bear
\frac {\partial f} {\partial t}(t, k)&=&E(f)\equiv -\Gamma (k)\,f(t, k)+T_2[f]\equiv T_1[f]+T_2[f]\label{E8bis}\\
T_2[f]&=&
2\int _0^\infty K(k, k') f(t, k')dk'  \label{E9}\\
\Gamma (k)&=&\sinh k\int _0^\infty \!\!\! \left(\phi (|k-k'|)\phi (k')+\phi (k+k')\phi (k') \right)dk'  \label{E10}\\
K(k, k')&=&\left(\phi (|k-k'|)-\phi (k+k') \right)k\, k' \label{E11}\\
\phi (k)&=&\frac {k^2} {\sinh k}\label{E12}\\
f(0)&=&f_0\label{E12b}.
\eear

The function $\Gamma (k)$ defined by (\ref{E10}) is such that:
\bear
\Gamma (k)\sim \frac {k^5} {15},\,\,\,\hbox{as}\,\,k\to +\infty \label{E14}\\
\Gamma (k)\sim \frac {\pi ^4 k} {15},\,\,\,\hbox{as}\,\,k\to 0, \label{E15}
\eear
(cf. \cite{Buot:ORT:1997} and Appendix below), and then, its range is $[0, +\infty)$.  

We introduce the following auxiliary function that will be needed in all the sequel:
\begin{equation}
\varphi _0=\frac {\phi } {||\phi ||_2}=\frac {\sqrt{30} } {\pi^2}\varphi .\label{S2E12}
\end{equation}

Then, Theorem  (\ref{theorem}) is a consequence of the following result.
\begin{theorem}\label{BoltzmannPhononTheoremExponentialDecay} Suppose that $f_0\in L^2(\mathbb{R}_+)$ and denote
\begin{equation}
c_0=\int _0^\infty f_0(k)\varphi _0(k)dk,\label{defczero}
\end{equation}
where $\varphi _0$ is defined in (\ref{S2E12}). Then,

(i)  there exists a  unique function  $f$ such that
\begin{eqnarray}
&&(f-c_0\varphi _0)\in L^2((0,\infty),  L^2(\Gamma )),\label{propsol0} \\
&&f\in L^\infty((0,\infty), L^2(\mathbb{R}_+))\cap C([0,\infty), L^2(\mathbb{R}_+)),\label{propsol1}\\
&& \frac {\partial f} {\partial t}\in L^2(0, \infty;  L^2(\Gamma ^{-1})),\label{propsol2}
\end{eqnarray}
that satisfies the equation  (\ref{E8bis}) in $L^2((0, \infty), L^2(\Gamma ^{-1}))$ and takes the initial data $f_0$ in the following sense:
\begin{eqnarray}
\label{propsol30}
\lim _{ t\to 0 }\left(||f(t)-f_0|| _{ L^2(\Gamma ^{-1}) }+||f(t)-f_0|| _{2}\right)=0.
\end{eqnarray}
This solution also satisfies 
\begin{eqnarray}
\label{propsol3000}
||f(t)||_2^2+2C_*\int _0^\infty||f(t)-c_0\varphi _0|| ^2_{ L^2(\Gamma ) }dt\le 2||f_0||_2^2
\end{eqnarray}
\bear
\label{propsol3250}
\left|\left|  \frac {\partial f} {\partial t} \right|\right| _{ L^2(0, \infty;  L^2(\Gamma ^{-1})) }\le 
 (1+2C_0)||f|| _{ L^2(0, \infty;  L^2(\Gamma ))  },
\eear
for some constant $C_0>0$, and the conservation of energy:
\begin{equation}
\label{consenergy}
\forall t>0:\,\,\,\,\frac {d} {dt}\int _0^\infty\!\!\! f(t, k)\frac {k^2\, dk} {\sinh(k)}=0.
\end{equation}
If $f_0\ge 0$, then $f(t, k)\ge 0$ for all $t>0$ and a. e. $k>0$.

(ii) If $f_0$ also satisfies one of the two following conditions:
\begin{equation}\label{BoltzmannPhononTheoremExponentialDecayCondition}
I=\int_{0}^1\frac{|f_0(k)|^2}{k}dk<\infty
\end{equation}
\begin{equation}\label{BoltzmannPhononTheoremExponentialDecayCondition2}
a= \lim _{ k\to 0 }f_0(k)\,\,\,\hbox{exists}.
\end{equation}
there exists a positive constant $C$, depending on $I$ or $a$ respectively,  such that, for all $t>0$:
\begin{equation}\label{BoltzmannPhononTheoremExponentialDecayExDecay}
\|f(t)-c_0\varphi _0 \|_{2}\leq C\frac{||f_0-c_0\varphi _0||_2}{(1+t)^{1/2}}.
\end{equation}
where $\varphi _0$ is defined in (\ref{S2E12})  and $c_0$ is given by (\ref{defczero}).
\end{theorem}
 
The algebraic rate of convergence in $L^2(\RR_+)$ norm is proved using  classical arguments. We first establish a coercivity property of the operator $E$ in a suitable functional space. Then,  this coercivity is used to obtain an upper estimate of the convergence rate. This last step uses the detailed behavior of the kernel $K$ and the function $\Gamma $ near  $k=0$. 

The plan of the paper is the following. We prove in Section \ref{SectionOperatorE} two important properties of the operator $E$.
Section \ref{SectionExistence} is devoted to the  proof of an existence and uniqueness result for the solution of Cauchy problem (\ref{E8bis})-(\ref{E12b}). In Section \ref{Rate} we prove the convergence rate of the solutions of the problem (\ref{E8bis})--(\ref{E12b}).  
In Section \ref{nonradial} we prove  Proposition \ref{Prop1} and Theorem \ref{theorem}. We give in a final Appendix some auxiliary results, in particular the detailed behaviors of the functions $\Gamma $ and $K$.

\section{Properties of the operator $E$}
\label{SectionOperatorE}
In this  Section we prove several important properties of the operator $E$. We will be using  the following spaces.
\begin{eqnarray*}
&&L^2(\Gamma )=\left\{u: (0, \infty) \to \RR; \hbox{measurable, such that }\,||u|| _{ L^2(\Gamma ) }<\infty \right\}\\
&&L^2(\Gamma^{-1} )=\left\{u: (0, \infty) \to \RR; \hbox{measurable, such that }\,||u|| _{ L^2(\Gamma^{-1} ) }<\infty \right\}
\end{eqnarray*}
where
\begin{eqnarray*}
||u|| _{ L^2(\Gamma ) }=\left( \int_{0}^\infty |u(k)|^2\Gamma(k) dk\right)^{1/2}\\
||u|| _{ L^2(\Gamma ^{-1}) }=\left( \int_{0}^\infty \frac{|u(k)|^2}{\Gamma(k)} dk\right)^{1/2}.
\end{eqnarray*}
We shall also use the classical $L^2(\RR_+)$ of functions of integrable square in $(0, \infty)$, with its norm $||\cdot||_2$. 

Since several Hilbert spaces will be used all along this work, we want to be careful with the notation.  We denote by $\langle \cdot, \cdot \rangle$  the scalar product in $L^2(\RR_+)$:
$$
\langle \varphi  ,  \psi \rangle=\int _0^\infty \varphi (k)\psi (k)dk
$$
whenever this integral is well defined. We will also use the notation $\perp$ to denote the orthogonality with respect to the scalar product of $L^2(\RR_+)$:
$$
\varphi \perp \psi  \Longleftrightarrow \int _0^\infty \varphi (k)\psi (k)dk=0
$$
and similarly, if $A$ is a set of measurable functions,
$$
\varphi \in A^{\perp}\Longleftrightarrow \int _0^\infty \varphi (k)\psi (k)dk=0,\,\,\,\forall \psi \in A.
$$
We may then have $\varphi \perp \psi $ even if neither $\varphi $ nor $\psi $ belong to $L^2(\RR_+)$, as long as the integral on the right hand side is well defined and equal to zero.
\begin{lemma}\label{Lemma2}
The operator $E$ defined by (\ref{E8bis})--(\ref{E12}) is linear and continuous from $L^2(\Gamma )$ into 
$L^2(\Gamma ^{-1})$ and, for every $u\in L^2(\Gamma )$:
\begin{equation}
||E(u)|| _{ L^2(\Gamma ^{-1}) }\le (1+2C_0)||u|| _{ L^2(\Gamma ) } \label{Lemma2E10}
\end{equation}
where
\begin{equation}
C_0= \left(\int _0^\infty \left|\frac{K(k, k')}{\sqrt {\Gamma (k) \Gamma (k')}}\right|^2dk'dk\right)^{1/2}<\infty. \label{Lemma2E1}
\end{equation}
\end{lemma}
\begin{proof}
The proof that the integral defining $C_0$ in  (\ref{Lemma2E1}) converges is given in detail in the Appendix. On the other hand, for all $u\in L^2(\Gamma )$ and $v\in L^2(\Gamma )$:
\begin{eqnarray*}
&&\langle E(u), v\rangle=-\int _0^\infty \Gamma (k)u(k)v(k)dk+2\int _0^\infty \int _0^\infty K(k, k')u(k')v(k)dk'dk\\
&&\left|\int _0^\infty \Gamma (k)u(k)v(k)dk\right|\le \left(\int _0^\infty \Gamma (k)|u(k)|^2dk\right)^{1/2}\left(\int _0^\infty \Gamma (k)|v(k)|^2dk\right)^{1/2}\\
&&\left|\int _0^\infty \int _0^\infty K(k, k')u(k')v(k)dk'dk\right|=\\
&&\hskip 3.5cm =\left|\int _0^\infty \int _0^\infty \frac{K(k, k')}{\sqrt {\Gamma (k) \Gamma (k')}}\sqrt {\Gamma (k) \Gamma (k')}u(k')v(k)dk'dk\right|\\
&&= \int _0^\infty \left|\sqrt{\Gamma (k)}v(k)\int _0^\infty \frac{K(k, k')}{\sqrt {\Gamma (k) \Gamma (k')}}\sqrt {\Gamma (k')}u(k')dk'\right|dk\\
&&\le \left(\int _0^\infty \Gamma (k)|v(k)|^2dk\right)^{1/2}  \left(\int _0^\infty \left|\frac{K(k, k')}{\sqrt {\Gamma (k) \Gamma (k')}}\sqrt {\Gamma (k')}u(k')dk'\right|dk\right)^{1/2}\\
&&\le \left(\int _0^\infty \Gamma (k)|v(k)|^2\right)^{1/2} \left(\int _0^\infty \Gamma (k')|u(k')|^2dk'\right)^{1/2}\times \\
&&\hskip 6cm 
\times \left(\int _0^\infty \int _0^\infty \left|\frac{K(k, k')}{\sqrt {\Gamma (k) \Gamma (k')}}\right|^2dk'dk\right)^{1/2}.
\end{eqnarray*}
We have then for all $u\in L^2(\Gamma )$ and $v\in L^2(\Gamma )$:
\begin{equation*}
|\langle E(u), v\rangle|\le (1+2C_0)||u|| _{ L^2(\Gamma ) }||v|| _{ L^2(\Gamma ) }.
\end{equation*}
from where $E(u)\in (L^2(\Gamma ))'=L^2(\Gamma ^{-1})$ and (\ref{Lemma2E10}) follows.
\end{proof}

It was already shown in \cite{Buot:ORT:1997} that the operator $E$ is non negative. The precise property and its proof are  given in the following Lemma for  the sake of completeness.
\begin{lemma}\label{Lemma20}
For all $f\in L^2(\Gamma )$ and $g\in L^2(\Gamma )$:
\begin{eqnarray}
&&\langle - E(f), g\rangle = \int_0^\infty\int_0^\infty\phi(k+k')\phi(k')\phi(k)\times \label{conservation10}\\
& &\times\left[\frac{\sinh(k)f(k)}{k}+\frac{\sinh(k')f(k')}{k'}-\frac{\sinh(k+k')f(k+k')}{k+k'}\right]\times\nonumber\\
& &\times\left[\frac{\sinh(k)g(k)}{k}+\frac{\sinh(k')g(k')}{k'}-\frac{\sinh(k+k')g(k+k')}{k+k'}\right]\!dkdk'.\nonumber
\end{eqnarray}
\end{lemma}
\begin{proof}
We first notice that by definition:
\begin{eqnarray*}
\langle -Ef,g\rangle&=&\int_0^\infty\int_0^\infty\sinh k\left(\phi (|k-k'|)\phi (k')+\phi (k+k')\phi (k') \right)f(k)g(k)dk' dk\\
& &-2\int_0^\infty\int_0^\infty\left(\phi (|k-k'|)-\phi (k+k') \right)k\, k' f(k')g(k)dk'dk\\
&=& I_1+I_2+I_3+I_4.
\end{eqnarray*}
We now write the integrals $I_1$, $I_2$, $I_3$ and $I_4$ using the definitions and symmetries of the two  functions $\Gamma (k)$ and $K(k, k')$.
\begin{eqnarray*}
&&I_1=\int_0^\infty\int_0^\infty\sinh k\left(\frac{|k-k'|^2}{\sinh (|k-k'|)}\frac{|k'|^2}{\sinh (k')} \right)f(k)g(k)dk' dk\\
&&I_2=\int_0^\infty\int_0^\infty\sinh k\left(\frac{|k+k'|^2}{\sinh (|k+k'|)}\frac{|k'|^2}{\sinh (k')} \right)f(k)g(k)dk' dk\\
&&I_3= -2\int_0^\infty\int_0^\infty\left(\frac{|k-k'|^2}{\sinh (|k-k'|)} \right)k\, k' f(k')g(k)dk'dk\\
&&I_4= 2\int_0^\infty\int_0^\infty\left(\frac{|k+k'|^2}{\sinh (|k+k'|)} \right)k\, k' f(k')g(k)dk'dk.
\end{eqnarray*}
Let us denote for the remaining of this calculation $Q[g ](k)=\frac{\sinh (k)g (k)}{k}$
\begin{eqnarray}\label{pos1}\nonumber
&&I_1
= \int_0^\infty\int_0^\infty\frac{|k-k'|^2}{\sinh (|k-k'|)}\frac{|k'|^2}{\sinh (k')}\frac{|k|^2}{\sinh (k)}Q[f](k)Q[g](k)dk'dk \nonumber\\
&&= \int_{\{k>k'\}}\frac{|k-k'|^2}{\sinh (|k-k'|)}\frac{|k'|^2}{\sinh (k')}\frac{|k|^2}{\sinh (k)}Q[f](k)Q[g](k)dk'dk \nonumber \\
&&+\int_{\{k<k'\}}\frac{|k-k'|^2}{\sinh (|k-k'|)}\frac{|k'|^2}{\sinh (k')}\frac{|k|^2}{\sinh (k)}Q[f](k)Q[g](k)dk'dk \nonumber \\
& &=\int_{0}^\infty\int_0^\infty\phi (k)\phi (k')\phi (k+k')Q[f](k+k')Q[g](k+k')+ \nonumber \\
&&+\int_{0}^\infty\int_0^\infty\phi (k)\phi (k')\phi (k+k')Q[f](k)Q[g](k)dk'dk.
\end{eqnarray}
\begin{eqnarray}\label{pos2}
&&I_2 =\int_0^\infty\int_0^\infty\left(\frac{|k+k'|^2}{\sinh (|k+k'|)}\frac{|k'|^2}{\sinh (k')} \frac{|k|^2}{\sinh (k)} \right)\frac{\sinh (k)^2}{|k|^2}f(k)g(k)dk' dk\nonumber \\
&& =\int_0^\infty\int_0^\infty\phi(k+k')\phi(k')\phi(k)Q[f](k)Q[g](k)dk' dk \nonumber \\
&& =\int_0^\infty\int_0^\infty\phi(k+k')\phi(k')\phi(k)Q[f](k')Q[g](k')dk' dk.
\end{eqnarray}
\begin{eqnarray}\label{pos3}\nonumber
&&I_3
=-2\int_{\{k>k'\}}\left(\frac{|k-k'|^2}{\sinh (|k-k'|)} \right)k\, k' f(k')g(k)dk'dk\nonumber\\
& &-2\int_{\{k>k'\}}\left(\frac{|k-k'|^2}{\sinh (|k-k'|)} \right)k\, k' f(k)g(k')dk'dk\nonumber\\
&&= -2\int_0^\infty\int_0^\infty\left(\frac{|k|^2}{\sinh (|k|)} \right)(k+k')\, k' f(k')g(k+k')dk'dk\nonumber\\
& &-2\int_0^\infty\int_0^\infty\left(\frac{|k|^2}{\sinh (|k|)} \right)(k+k')\, k' g(k')f(k+k')dk'dk\nonumber\\
&&=-\int_0^\infty\int_0^\infty\phi(k+k')\phi(k')\phi(k)\left(Q[f](k)+Q[f](k')\right)Q[g](k+k')dkdk'\nonumber\\
&&-\int_0^\infty\int_0^\infty\phi(k+k')\phi(k')\phi(k)\left(Q[g](k)+Q[g](k')\right)Q[f](k+k')dkdk'.\nonumber\\
\end{eqnarray}
\begin{eqnarray}\label{pos4}\nonumber
&&I_4
=\int_0^\infty\int_0^\infty\phi(k+k')\phi(k')\phi(k)Q[f](k)Q[g](k')dk' dk+ \nonumber\\
&&+\int_0^\infty\int_0^\infty\phi(k+k')\phi(k')\phi(k)Q[f](k')Q[g](k)dk' dk.
\end{eqnarray}
Identity  (\ref{conservation10}) follows by combining (\ref{pos1})--(\ref{pos4}). 

\end{proof}

\begin{corollary}
\label{coro} Let $\phi $ be the function defined in (\ref{E12}). Then 
\label{zerovp}
\begin{equation}
\label{eigenval}
E(\phi ) =0.
\end{equation}
Conversely, if $f\in L^2(\Gamma )$ is such that $ E(f)=0$, then $f=C \phi $ for some constant $C$.
\end{corollary}

\begin{proof}
By (\ref{conservation10}), $\langle E(\phi ), g\rangle=0$ for all $g\in L^2(\Gamma )$ and (\ref{zerovp}) follows.  On the other hand, if $ E(f) =0$ for some $f \in L^2(\Gamma )$, then $\langle E(f), f\rangle=0$ and by  (\ref{conservation10}):
$$
\left[\frac{\sinh(k)f(k)}{k}+\frac{\sinh(k')f(k')}{k'}-\frac{\sinh(k+k')f(k+k')}{k+k'}\right]^2=0
$$
for a. e. $k>0$, $k'>0$. The function $\frac{\sinh(k)f(k)}{k}$ must then be linear, and we must then have $f=C \phi $ for some positive constant $C$.
\end{proof}

\begin{corollary}
\label{S2Cor2}
For all $f\in L^2(\Gamma )$ and $g\in L^2(\Gamma )$:
\begin{equation}
\label{S2Cor2E1}
\left|\langle -E(f), g \rangle \right| \le \frac {1} {2}\langle -E(f), f \rangle+\frac {1} {2}\langle -E(g), g\rangle.
\end{equation}
\end{corollary}
\begin{proof}
By (\ref{conservation10}) in Lemma \ref{Lemma20}:

$$
\left|\langle - E(f), g\rangle\right| \le  \int_0^\infty\int_0^\infty|q(f)(k, k')q(g)(k, k')| d\mu (k, k').
$$
where
$$
q(h)(k, k')=\frac{\sinh(k)h (k)}{k}+\frac{\sinh(k')h (k')}{k'}-\frac{\sinh(k+k')h (k+k')}{k+k'}
$$
and
$$
d\mu =\phi(k+k')\phi(k')\phi(k) dkdk'
$$
is a non negative measure. We deduce by Holder's inequality
\begin{eqnarray*}
\left|\langle - E(f), g\rangle\right| &\le& \frac {1} {2}\int_0^\infty\int_0^\infty |q(f)(k, k')|^2\!d\mu  +\\
&&+\frac {1} {2}\int_0^\infty\int_0^\infty|q(g)(k, k')|^2\!d\mu \\
&=&\frac {1} {2}\langle -E(f), f \rangle+\frac {1} {2}\langle -E(g), g\rangle.
\end{eqnarray*}

\end{proof}
As we have seen, the operator $E$ is continuous from  $ L^2(\Gamma )$ into $L^2(\Gamma ^{-1})$. By the Corollary \ref{zerovp}, its kernel, $N(E)$ is a one dimensional vector space generated by the function $\phi $. 

\begin{lemma}\label{LemmaCoercivity}
There exists a constant $C_*>0$ such that, for all $h\in L^2(\Gamma )$:
\begin{equation}
\label{SpectralGap1}
\langle -Eh, h\rangle \geq C_*\|h-\mathbb{P}h\|^2_{L^2(\Gamma )},
\end{equation}
where 
$$
\mathbb{P}h=c_0(h)\varphi _0,\,\,\,c_0(h)=\int _0^\infty h(k)\varphi _0(k)dk.
$$
\end{lemma}
\begin{remark}
The map $\mathbb{P}$  is the orthogonal projection  on the kernel $N(E)$ for the scalar product of $L^2(\RR_+)$. Since $\varphi _0\in L^2(\Gamma ^{-1})$ it is well defined for all $h \in L^2(\Gamma )$.
\end{remark}
\begin{proof}
For all $h\in L^2(\Gamma )$, we denote
\begin{eqnarray*}
h=c_0(h)\varphi _0+g,\,\,\,c_0(h)\varphi_0 =\mathbb{P}h \in N(E).
\end{eqnarray*}
Notice that,
$$
\int _0^\infty g(k)\varphi _0(k)dk=\int _0^\infty (h(k)-c_0(h)\varphi _0(k))\varphi _0(k)dk=0
$$
and so $g\in N(E)^\perp$.
Moreover, by Lemma \ref{Lemma20}, we deduce that
$$
\langle E(g), \mathbb{P}h\rangle=c_0(h)\langle E(g), \varphi _0\rangle=0
$$
and then,
$$
\langle Eh, h\rangle=\langle E(g), \mathbb{P}h+g\rangle=\langle E(g), g\rangle.
$$
Therefore, property (\ref{SpectralGap1}) is equivalent to
\begin{equation}
\label{SpectralGap1bis}
\forall g\in L^2(\Gamma ),\,\mathbb{P}g=0:\,\,\,\,\langle -Eg, g\rangle \geq C_*\|g\|^2_{L^2(\Gamma )}.
\end{equation}
In order to prove (\ref{SpectralGap1bis}), we show that for all $h\in L^2(\Gamma )$:
\begin{eqnarray}\label{SpectralGap1b}
&&\langle -Eh, h\rangle+c_0^2(h)\geq C_*\|h\|^2_{L^2(\Gamma )}.
\end{eqnarray}
To this end we make a change of unknown variable and define $g=\alpha h$, with $\alpha =\sqrt \Gamma $. The problem is now equivalent to prove that for all $g\in L^2(\RR_+)$:
\begin{eqnarray}\label{SpectralGap2}
& &\int _0^\infty|g(k)|^2dk-2\int _0^\infty \int _0^\infty {\frac{K(k,k')}{\alpha (k)\alpha (k')}}g(k')g(k)dk'dk\\\nonumber
& &+\int _0^\infty \int _0^\infty {\frac{ \varphi _0(k)\varphi _0(k')}{\alpha (k)\alpha (k')}}g(k')g(k)dk'dk\geq C_*\|g\|^2_{L^2}.
\end{eqnarray}
This follows from simple spectral properties of the operator $\widetilde E=-I+T$ with 
$$T: g\to \int _0^\infty {\frac{2K(k,k')}{\alpha (k)\alpha (k')}}g(k')dk'-\int _0^\infty {\frac{ \varphi _0(k)\varphi _0(k')}{\alpha (k)\alpha (k')}}g(k')dk'.$$ 
Since the two functions $\frac{2K(k,k')}{\alpha (k)\alpha (k')}$ and $\frac{ \varphi _0(k)\varphi _0(k')}{\alpha (k)\alpha (k')}$ belong to $L^2(\RR_+\times \RR_+)$, (for the first function this is proved  in detail in Lemma \ref{kernelK} of the Appendix), the operator $T$ is a Hilbert Schmidt, and then a compact, operator from $L^2(\RR_+)$ into itself. Its spectrum is then reduced to a sequence $(\mu _j) _{ j\in \NN }$ of eigenvalues satisfying $\mu _j\to 0$ as $j\to \infty$. The spectrum of $-\widetilde E$ is then also reduced to a sequence $(\lambda  _j) _{ j\in \NN }$ of eigenvalues  such that $\lambda  _j\to 1$ as $j\to \infty$. Since the operator $-E$  is non negative on $L^2(\Gamma )$ it is easy to deduce that $-\widetilde E$ is non negative on $L^2(\RR_+)$, and then $\lambda  _j\ge 0$ for all $j\in \NN$. In order to prove (\ref{SpectralGap2}) we then only need to show that zero is not an eigenvalue of $-\widetilde E$. If that was the case, any associated eigenfunction
  $g\in L
 ^2(\RR_+)$ would satisfy $\widetilde E (g)=0$ and then, multiplying by $g$ and integrating 
$$
-\left\langle  E \left(\frac {g} {\alpha }\right), \frac {g} {\alpha }\right \rangle+\int _0^\infty \int _0^\infty {\frac{ \varphi _0(k)\varphi _0(k')}{\alpha (k)\alpha (k')}}g(k')g(k)dk'dk=0.
$$
But this would imply that, for the function $h=\frac {g} {\alpha }\in L^2(\Gamma )$, we have
$$
\langle -Eh, h\rangle+c_0(h)^2=0.
$$
Since $\langle -Eh, h\rangle \ge 0$ this implies that $\langle -Eh, h\rangle=0$ and $c_0(h)=0$. By Corollary \ref{zerovp}, the first condition implies that $h\in N(E)$. Then we deduce from the second that $h=0$ and then $g=0$. This proves that zero is not an eigenvalue of $\widetilde E$ and we deduce that
$$
C_*=\min  _{ j\in \NN }\lambda _j>0.
$$
Property  (\ref{SpectralGap2}) follows, and then also  (\ref{SpectralGap1b}) for all $g\in L^2(\Gamma )$ and   (\ref{SpectralGap1bis}) for all $g\in L^2(\Gamma )$ such that $\mathbb{P}g=0$. This concludes the proof of (\ref{SpectralGap1}).
\end{proof}

\section{Existence and uniqueness of global solution.}
\label{SectionExistence}
In this Section we prove that the Cauchy problem (\ref{E8bis})-(\ref{E12b}) is well posed in $L^2(\RR_+)$. More precisely, we have the following proposition that is the first part of Theorem $\ref{BoltzmannPhononTheoremExponentialDecay}$.
\begin{proposition}
\label{TheoremExistence} 
Suppose that  $f_0\in L^2(\mathbb{R}_+)$.  Then,  the problem  (\ref{E8bis})--(\ref{E12b}) has a unique solution  $f$ such that
\begin{eqnarray}
&&(f-\mathbb{P}(f_0))\in L^2((0,\infty),  L^2(\Gamma )),\label{propsol0} \\
&&f\in L^\infty((0,\infty), L^2(\mathbb{R}_+))\cap C([0,\infty), L^2(\mathbb{R}_+)),\label{propsol1}\\
&&\partial _t f\in L^2((0, \infty), L^2(\Gamma ^{-1})),\label{propsol2}
\end{eqnarray}
that satisfies the equation  (\ref{E8bis}) in $L^2((0, T); L^2(\Gamma ^{-1}))$ for all $T>0$ and takes the initial data in the following sense:
\begin{eqnarray}
\label{propsol30}
\lim _{ t\to 0 }\left(||f(t)-f_0|| _{ L^2(\Gamma ^{-1}) }+||f(t)-f_0|| _{2}\right)=0.
\end{eqnarray}
This solution is such that, for all $\varphi \in L^2(\Gamma )$:
\begin{eqnarray}
\label{conservation1}
&&\frac {d} {dt}\int _0^\infty f(t, k)\varphi  (k)dk=\int_0^\infty\int_0^\infty\phi(k+k')\phi(k')\phi(k)\times \\
& &\times\left[\frac{\sinh(k)f(k)}{k}+\frac{\sinh(k')f(k')}{k'}-\frac{\sinh(k+k')f(k+k')}{k+k'}\right]\times\nonumber\\
& &\times\left[\frac{\sinh(k)\varphi (k)}{k}+\frac{\sinh(k')\varphi (k')}{k'}-\frac{\sinh(k+k')\varphi (k+k')}{k+k'}\right]\!dkdk'.\nonumber
\end{eqnarray}
In particular, for all $t>0$:
\begin{eqnarray}
\label{conservation2}
\frac {d} {dt}\int _0^\infty\!\!\! f(t, k)\frac {k^2 dk} {\sinh(k)}=0.
\end{eqnarray}
Moreover, for all $t>0$:
\begin{eqnarray}
\label{propsol3}
||f(t)||_2^2+2C_*\int _0^\infty||f(t)-\mathbb{P}(f_0)|| ^2_{ L^2(\Gamma ) }dt\le 2||f_0||_2^2 .
\end{eqnarray}
and
\bear
\label{propsol3250We}
\left|\left|  \frac {\partial f} {\partial t} \right|\right|_{ L^2(0, \infty;  L^2(\Gamma ^{-1})) }\le 
 (1+2C_0)||f-\mathbb{P}(f_0)||_{ L^2(0, \infty;  L^2(\Gamma ))  },
\eear
where the  constant $C_0$ is defined in (\ref{Lemma2E1}).

\noindent
If $f_0\ge 0$ then for all $t>0$, $f(t, k)\ge 0$ for a.e. $k>0$.
\end{proposition}

\begin{proof} \textit{Step 1: Uniqueness.}  We first prove that if there is a solution of  (\ref{E8bis})-(\ref{E12b}) satisfying  (\ref{propsol1})-(\ref{propsol2}), then it is unique.  Since the equation is linear it is sufficient to prove that the only solution  of  (\ref{E8bis})-(\ref{E12b}) satisfying  (\ref{propsol1})-(\ref{propsol2}) with initial data $f_0=0$ is the function such that $f(t)=0$ for all $t>0$. To this end, we multiply the equation (\ref{E8bis}) by $f$ and integrate on $k>0$ to obtain:
$$
\frac {1} {2}\frac {d} {dt}||f(t)||_2^2=\langle E(f), f\rangle.
$$
Since $c_0=0$ by hypothesis, we deduce using (\ref{SpectralGap1b}):
$$
\frac {1} {2}\frac {d} {dt}||f(t)||_2^2\le -C_*||f(t)||^2 _{ L^2(\Gamma ) }.
$$
If we now integrate this in time:
$$
||f(t)||_2^2+2C_*\int _0^t||f(s)||^2 _{ L^2(\Gamma ) }ds\le 0
$$
since $||f(0)||_2^2=||f_0||_2^2$ by the continuity of the application $t\mapsto ||f(t)||_2$ and uniqueness then follows.
\\\textit{Step 2.} We define the following truncation of the operator $E$ and the initial data $f_0$:
\begin{eqnarray}
&&E_n[h]\equiv -\Gamma_n (k)\,h(k)+T_{n2}[h],\label{Existence1}\\
&&\Gamma_n (k)=\Gamma (k)\chi_{n}(k),\label{Existence3}\\
&&T_{n2}[h]=2\int _0^\infty K_n(k, k') h(k')dk',  \label{Existence2}\\
&&K_n(k, k')=\chi_{n}(k) \chi_{n}(k')K(k, k'),\label{Existence4}\\
&&f_{0, n} (k)=\chi_{n}(k)f_0(k). \label{Existence6}\\
&&\chi_n(k)=\chi_{\{1/n<|k|<n\}}\nonumber,
\end{eqnarray}
where $\chi_{ A}$ is the characteristic function of the set $A$.

For every $n\in \NN$,  $E_n$ is now a linear and bounded operator  from $L^2(\RR_+)$ into itself.  Therefore, the linear problem
\begin{eqnarray}
&&\frac {\partial f} {\partial t}(t, k)=E_n(f)(t, k),\,\,\,t>0, k>0,\label{truncequation1} \\
&&f(0, k)=f_{0, n} (k),\,\,\,k>0,\label{truncequation2}
\end{eqnarray}
has a solution:
\begin{eqnarray*}
\label{truncsolution1}
f_n(t, k)=e^{t E_n}(f _{ 0, n })
\end{eqnarray*}
satisfying
\begin{eqnarray}
\label{truncsolution2}
f_n\in C([0, \infty); L^2(\RR_+))\cap C^\infty (0, \infty; L^2(\RR_+)).
\end{eqnarray}

The same argument as in Step 1 shows that $f_n$ is unique. If moreover $f_0\ge 0$, then $f _{ 0, n }\ge 0$ and then $f_n(t)\ge 0$ for all $t>0$. 

Since $supp(f_{0, n})\subset (1/n,n)$, $supp(\Gamma_n)\subset (1/n,n)$ and $supp(K_n)\subset (1/n,n)\times (1/n,n)$, we have
$supp(f_n(t))\subset (1/n,n)$ for all $t>0$, and therefore:
$$E_nf_n=(Ef_n)\chi_{n}(k),$$
where $\chi_{ (1/n,n)}$ is the characteristic function of ${ (1/n,n)}$. 
The function $f_n$ solves then:
\begin{eqnarray}
\frac {\partial f_n} {\partial t}=(Ef_n)\chi_{n}(k)\label{nequation}\\
f_n(0)=f _{ 0, n }\label{ndata}
\end{eqnarray}
Multiplying (\ref{nequation})  by $f_n$ we obtain,  after integration on $(0, t)\times \RR_+$:
\begin{equation}
\label{S2E1}
\|f_n(t)\|_2^2-\|f_{0, n}\|_2^2=2\int_0^t\langle (Ef_n(s))\chi_{n}, f_n(s)\rangle dt.
\end{equation}
Since for all $s>0$ $supp(f_n(s))\subset (1/n,n)$ we notice first that:
$$
\langle (Ef_n(s))\chi_{n}, f_n(s)\rangle=\langle Ef_n(s), f_n(s)\rangle
$$
and second, that  $f_n(t)\in L^2(\Gamma )$. Then, by (\ref{SpectralGap1b}) in the proof of Lemma \ref{LemmaCoercivity}, we deduce, for all $T>0$: 
\begin{eqnarray}
||f_n(T)||_2^2+2C_*\int_0^T||f_n(t)-\mathbb{P}f_n(t)||^2 _{ L^2(\Gamma ) }dt & \le & \|f_{0, n}\|_2^2.\label{S3E10}
\end{eqnarray}
It first follows from (\ref{S3E10}) that for all $t\ge 0$
\begin{equation}
||f_n(t)||_2^2 \le  \|f_{0}\|_2^2.\label{S3E1000}
\end{equation}
Using (\ref{S3E1000}) we obtain that for all $t>0$:
\begin{eqnarray*}
||\mathbb{P}f_n(t)||^2 _{ L^2(\Gamma ) }&=&\left(\int _0^\infty f_n(t, k)\varphi _0(k)dk\right)^2||\varphi _0||^2 _{ L^2(\Gamma ) }\\
&\le & ||f_n(t)||_2^2 ||\varphi _0||^2 _{ L^2(\Gamma ) }\le  \|f_{0}\|_2^2||\varphi _0||^2 _{ L^2(\Gamma ) }
\end{eqnarray*}
and then, using this in (\ref{S3E10}):

\begin{eqnarray}\nonumber
&&||f_n(T)||_2^2+C_*\int_0^T||f_n(t)||^2 _{ L^2(\Gamma ) }dt\le \|f_{0}\|_2^2\left(1+2C_*T||\varphi_0||^2_{ L^2(\Gamma ) }\right)\label{S3E10b}\\
&&\int_0^T||f_n(t)||^2 _{ L^2(\Gamma ) }dt\le \frac {\|f_0\|_2^2} {C_*}\left(1+2C_*T||\varphi_0||^2_{ L^2(\Gamma ) }\right).\label{S3E11}
\end{eqnarray}

By (\ref{S3E11}), the sequence $(f_n) _{ n\in \NN }$ is then bounded in $L^2(0, T; L^2(\Gamma ))$ for all $T>0$.  We prove now that it is also a Cauchy sequence in that space.

To this end, let $n, m$ be two positive  integers such that for example $m>n$.  By  (\ref{truncequation1}):
\begin{eqnarray}
\frac {\partial } {\partial t}(f_n-f_m)&=&E_nf_n-E_mf_m\label{nmequation}\\
f_n(0)-f_m(0)&=&f _{ 0, n }-f _{ 0, m }\label{nmdata}
\end{eqnarray}
After multiplication by $f_n-f_m$ and integration over $(0, \infty)$ we deduce as usual
\begin{eqnarray}
\|f _{n}(t)-f _{m}(t)\|_2^2-\|f _{ 0, n }-f _{ 0, m }\|_2^2
=2\int_0^t\left(\left<E_nf_n,f _{n}-f _{m}\right>\right. \nonumber \\
\left.-\left<E_mf_m,f _{n}-f _{m}\right>\right)dt. \label{S3Ek1}
\end{eqnarray}
We decompose the function $f_m$ as follows:
\begin{eqnarray}
&&f_m(t, k)=f _{ m, n }(t, k)+\varphi  _{ m, n }(t, k)\label{S3Ek2}\\
&&f _{ m, n }(t, k)=f_m(t, k)\chi_n(k) \label{S3Ek3}\\
&&\varphi  _{ m, n }(t, k)=f _{ m }(t, k)\left(\chi _{ m }(k)-\chi_n(k)\right)\label{S3Ek4}
\end{eqnarray}
and use this to rewrite  the two  right hand side terms of (\ref{S3Ek1}).  We have first:
\begin{eqnarray}
&&\left<E_nf_n,f _{n}-f _{m}\right>=-\int_0^\infty\Gamma_n (k)f_n(k)(f _{n}(k)-f _{m}(k))dk+\label{S3Ek5}\\
&&+\int_{\mathbb{R}_+^2}K_n (k,k')f_n(k')(f _{n}(k)-f _{m}(k))(k)dk'dk=J_1+J_2. \nonumber
\end{eqnarray}
Since the supports of $f_n$ and $\varphi  _{ m, n }$ are disjoint we have:
\begin{eqnarray}
J_1&=&-\int_0^\infty\Gamma_n (k)f_n(k)(f _{n}(k)-f _{m}(k))dk\nonumber \\
&=&-\int_0^\infty\Gamma_n (k)f_n(k)(f _{n}(k)-f _{m, n}(k))dk\nonumber\\
&=&-\int_0^\infty\Gamma(k)f_n(k)(f _{n}(k)-f _{m, n}(k))dk\label{S3Ek6}
\end{eqnarray}
Using that for any $k'>0$ the supports of $K_n(\cdot, k')$ and $\varphi  _{ m, n }$ are also disjoints we obtain:
\begin{eqnarray}
J_2&=&\int_{\mathbb{R}_+^2}K_n (k,k')f_n(k')(f _{n}(k)-f _{m}(k))(k)dk'dk\nonumber \\
&=&\int_{\mathbb{R}_+^2}K_n (k,k')f_n(k')(f _{n}(k)-f _{m, n}(k))(k)dk'dk\nonumber\\
&=&\int_{\mathbb{R}_+^2}K (k,k')f_n(k')(f _{n}(k)-f _{m, n}(k))(k)dk'dk\label{S3Ek7}
\end{eqnarray}
By (\ref{S3Ek6}) and (\ref{S3Ek7}), we deduce from (\ref{S3Ek5}) that
\begin{equation}
\left<E_nf_n,f _{n}-f _{m}\right>=\left<E f_n, f _{n}-f _{m, n}\right>. \label{S3Ek8}
\end{equation}
On the other hand, 
\begin{eqnarray}
&&\left<E_mf_m,f _{n}-f _{m}\right>=-\int_{\mathbb{R}_+}\Gamma_m (k)f_m(k)(f _{n}(k)-f _{m}(k))dk+\nonumber\\
&&+\int_{\mathbb{R}_+^2}K_m (k,k')f_m(k')(f _{n}(k)-f _{m}(k))(k)dk'dk=L_1+L_2. \label{S3Ek9}
\end{eqnarray}
We have now
\begin{eqnarray}
L_1&=&-\int_0^\infty\Gamma_m (k)(f_{m, n}(k)+\varphi  _{ m, n }(k))(f _{n}(k)-f _{m, n}(k)-\varphi  _{ m, n }(k))dk\nonumber \\
&=&-\int_0^\infty\Gamma_m (k)f_{m, n}(k)(f _{n}(k)-f _{m, n}(k))dk+\nonumber \\
&&\hskip 4cm +\int_0^\infty\Gamma_m (k)f_{m, n}(k)\varphi  _{ m, n }(k)dk\nonumber \\
&&-\int_0^\infty\Gamma_m (k)\varphi  _{ m, n }(k)(f _{n}(k)-f _{m, n}(k))dk+\nonumber \\
&&\hskip 4cm +\int_0^\infty\Gamma_m (k)\varphi  _{ m, n }(k)\varphi  _{ m, n }(k)dk. \label{S3Ek10}
\end{eqnarray}
Using the properties of the support of the functions $f_n$, $f _{ m, n }$, $\Gamma _n$  and $\varphi  _{ m, n }$ we deduce as above that the second and third terms in the right hand side of (\ref{S3Ek10}) are zero, from where:
\begin{eqnarray}
L_1&=&-\int_0^\infty\Gamma (k)f_{m, n}(k)(f _{n}(k)-f _{m, n}(k))dk+\nonumber\\ 
&&\hskip 3cm +\int_0^\infty\Gamma (k)\varphi  _{ m, n }(k)\varphi  _{ m, n }(k)dk
\label{S3Ek11}
\end{eqnarray}
Consider now $L_2$, that may be written as follows:
\begin{eqnarray}
&&\hskip -0.5cm L_2= \int_{\mathbb{R}_+^2}\!\!K_m (k,k')(f_{m, n}(k')+\varphi  _{ m, n }(k')) (f _{n}(k)-f _{m, n}(k)-\varphi  _{ m, n }(k))dk'dk  \nonumber \\
&&=\int_{\mathbb{R}_+^2}K_m (k,k')f_{m, n}(k')(f _{n}(k)-f _{m, n}(k))dk'dk  \nonumber \\
&&\hskip 2cm - \int_{\mathbb{R}_+^2}K_m (k,k')f_{m, n}(k')\varphi  _{ m, n }(k)dk'dk+\nonumber \\
&&+\int_{\mathbb{R}_+^2}K_m (k,k')\varphi _{m, n}(k')(f _{n}(k)-f _{m, n}(k))dk'dk\nonumber \\
&&\hskip 2cm-\int_{\mathbb{R}_+^2}K_m (k,k')\varphi _{m, n}(k')\varphi _{m, n}(k)dk'dk.
\label{S3Ek12}
\end{eqnarray}
We rewrite $L_2$ as follows:
\begin{eqnarray}
L_2
&=&\int_{\mathbb{R}_+^2}K (k,k')f_{m, n}(k')(f _{n}(k)-f _{m, n}(k))dk'dk  \nonumber \\
&-&\int_{\mathbb{R}_+^2}K (k,k')\varphi _{m, n}(k')\varphi _{m, n}(k)dk'dk + R _{ m, n }(t),\label{S3Ek14}\\
R _{ m, n }(t)&=&\int_{\mathbb{R}_+^2}K_m (k,k')\varphi _{m, n}(k')(f _{n}(k)-f _{m, n}(k))dk'dk\nonumber \\
&&\hskip 1cm - \int_{\mathbb{R}_+^2}K_m (k,k')f_{m, n}(k')\varphi  _{ m, n }(k)dk'dk.\label{S3Ek15}
\end{eqnarray}
It follows from (\ref{S3Ek9}), (\ref{S3Ek11}) and (\ref{S3Ek14}) that:
\begin{eqnarray*}
&&\left<E_mf_m,f _{n}-f _{m}\right>=-\int_0^\infty\Gamma (k)f_{m, n}(k)(f _{n}(k)-f _{m, n}(k))dk+\nonumber\\ 
&&\hskip 4.8cm +\int_0^\infty\Gamma (k)\varphi  _{ m, n }(k)\varphi  _{ m, n }(k)dk+\nonumber \\
&&\hskip 3cm +\int_{\mathbb{R}_+^2}K (k,k')f_{m, n}(k')(f _{n}(k)-f _{m, n}(k))dk'dk  \nonumber \\
&&\hskip 3cm -\int_{\mathbb{R}_+^2}K (k,k')\varphi _{m, n}(k')\varphi _{m, n}(k)dk'dk + R _{ m, n }(t, k)
\end{eqnarray*}
and then
\begin{equation}
\left<E_mf_m,f _{n}-f _{m}\right>=\left<Ef_{m, n},f _{n}-f _{m}\right>-\left<E\varphi _{m, n}, \varphi  _{ m, n }\right>+R _{ m, n }(t).
\label{S3Ek16}
\end{equation}
We deduce, using (\ref{S3Ek8}) and (\ref{S3Ek16}) that
\begin{eqnarray}
\left<E_nf_n,f _{n}-f _{m}\right>-\left<E_mf_m,f _{n}-f _{m}\right>=\left<E(f_n-f_{m, n}),f _{n}-f _{m}\right>+\nonumber \\
+\left<E\varphi _{m, n}, \varphi  _{ m, n }\right>
+R _{ m, n }(t). \label{S3Ek17}
\end{eqnarray}
By (\ref{SpectralGap1}) in  Lemma  \ref{LemmaCoercivity} we deduce
\begin{eqnarray}
&&\left<E_nf_n,f _{n}-f _{m}\right>-\left<E_mf_m, f _{n}-f _{m}\right> \le -  C_*\|(\mathbb{I}-\mathbb{P})(f_n-f _{ m, n })\|^2_{L^2(\Gamma )}\nonumber \\
&&\hskip 4.2cm -C_*\|(\mathbb{I}-\mathbb{P})\varphi  _{ m, n }\|^2_{L^2(\Gamma )}+|R _{ m, n }(t)|, \label{S3Ek18}
\end{eqnarray}
where  $\mathbb{I}$ is the identity operator. 

On the other hand, since
$$
\|(\mathbb{I}-\mathbb{P})(f_n-f_m)\|^2_{L^2(\Gamma )}\le \|(\mathbb{I}-\mathbb{P})(f_n-f_{m, n})\|^2_{L^2(\Gamma )}+\|(\mathbb{I}-\mathbb{P})\varphi  _{ m, n }\|^2_{L^2(\Gamma )}
$$
it follows that
\begin{eqnarray}
\left<E_nf_n,f _{n}-f _{m}\right>-\left<E_mf_m, f _{n}-f _{m}\right> \le -C_*\|(\mathbb{I}-\mathbb{P})(f_n-f _{ m})\|^2_{L^2(\Gamma )} \nonumber \\
+|R _{ m, n }(t)|.
 \label{S3Ek19}
\end{eqnarray}
We now estimate $R _{ m, n }$ given by (\ref{S3Ek15}). Since $K(k, k')=K(k', k)$ for all $k>0$, $k'>0$ it is easy to check that this term may be written as follows
\begin{eqnarray}
&&R _{ m, n }(t)=\int_{\mathbb{R}_+^2}K(k, k')\chi _{ m }(k')(\chi _{ m }(k')-\chi_n(k'))f_m(k')f_n(k)dk'dk \nonumber \\
&&\hskip 0.6cm -2\int_{\mathbb{R}_+^2}K(k, k')\chi _{n}(k)(\chi _{ m }(k')-\chi_n(k'))f_m(k')f_m(k)dk'dk.\label{S3Ek20}
\end{eqnarray}
from where we deduce the following estimate:
\begin{eqnarray}
|R _{ m, n }(t)|&\le&\int_{\mathbb{R}_+^2}K(k, k')\chi _{ m }(k')(\chi _{ m }(k')-\chi_n(k'))|f_m(k')||f_n(k)|dk'dk \nonumber \\
&&+2\int_{\mathbb{R}_+^2}K(k, k')\chi _{ m }(k')(\chi _{n }(k')-\chi_m(k'))|f_m(k')||f_m(k)|dk'dk \nonumber \\
&\le & \rho  _{ n, m } \left(||f_n|| _{ L^2(\Gamma ) }||f_m|| _{ L^2(\Gamma ) }+2||f_m||^2 _{ L^2(\Gamma ) } \right)\label{S3Ek21}\\
\rho  _{ n, m }&=&\left\|\frac{K (k,k')(\chi _{ m }(k')-\chi_n(k'))}{\sqrt\Gamma(k)\sqrt\Gamma(k')}\right\|_{L^2(\mathbb{R}_+^2)}\label{S3Ek21bis}
\end{eqnarray}
Using now  that $\frac{K (k,k')}{\sqrt\Gamma(k)\sqrt\Gamma(k')}\in L^2(\mathbb{R}_+^2) $ and  the dominated convergence Theorem, it is easy to check  that
\begin{equation}
\label{S3Ek22}
\lim _{n\to \infty, m>n }\rho(n, m)=0
\end{equation}
Combining now (\ref{S3Ek1}) and   (\ref{S3Ek19}):
\begin{eqnarray}
\|f _{n}(t)-f _{m}(t)\|_2^2+2C_*\int _0^t\|(\mathbb{I}-\mathbb{P})(f_n-f _{ m})\|^2_{L^2(\Gamma )}ds\le\nonumber \\
\le \|f _{ 0, n }-f _{ 0, m }\|_2^2+\int _0^t|R _{ m, n }(s)|ds.\label{S3Ek23}
\end{eqnarray}
On the other hand, since
\begin{eqnarray*}
||\mathbb{P}(f _{n}(t)-f _{m}(t))||^2 _{ L^2(\Gamma ) }
&\le & ||f_n(t)-f _{m}(t)||_2^2 ||\varphi _0||^2 _{ L^2(\Gamma ) },
\end{eqnarray*}
we have by  \eqref{S3Ek23}:
\begin{equation*}
||\mathbb{P}(f _{n}(t)-f _{m}(t))||^2 _{ L^2(\Gamma ) }
\le  \left(\|f _{ 0, n }-f _{ 0, m }\|_2^2 +\!\!\int _0^t|R _{ m, n }(s)|ds\right)||\varphi _0||^2 _{ L^2(\Gamma ) }.
\end{equation*}
Integrating both sides of this inequality with respect to $t$, we deduce
\begin{eqnarray*}
&&\int_0^t||\mathbb{P}(f _{n}(t)-f _{m}(t))||^2 _{ L^2(\Gamma ) }ds \le  t\left(\|f _{ 0, n }-f _{ 0, m }\|_2^2 +\right. \\
&&\hskip 6cm \left. +\int _0^t|R _{ m, n }(s)|ds\right)||\varphi _0||^2 _{ L^2(\Gamma ) },
\end{eqnarray*}
and then,
\begin{eqnarray}
&&\|f _{n}(t)-f _{m}(t)\|_2^2+2C_*\!\!\int _0^t\!\!\|(f_n-f _{ m})\|^2_{L^2(\Gamma )}ds\le (1+2C_*t||\varphi _0|| _{ L^2(\Gamma ) })\times \nonumber \\
&&\hskip 4cm \times\left(\|f _{ 0, n }-f _{ 0, m }\|_2^2+\int _0^t|R _{ m, n }(s)|ds\right). \label{S3Ek24}
\end{eqnarray}
By (\ref{S3Ek21}),
\begin{eqnarray}
\int _0^t|R _{ m, n }(s)|ds\le 3\rho  _{ n, m }\int _0^t \left(||f_n||^2 _{ L^2(\Gamma ) }+||f_m||^2 _{ L^2(\Gamma ) } \right).\label{S3Ek25}
\end{eqnarray}
Since the sequence $(f_n) _{ n\in \NN }$ is bounded in $L^2(0, T; L^2(\Gamma ))$ for all $T>0$ and $\rho  _{ n, m }$ satisfies (\ref{S3Ek22}), we deduce that   $(f_n) _{ n\in \NN }$ is a Cauchy sequence in $L^2(0, T; L^2(\Gamma ))$ for all $T>0$.

Then, there exists $f\in L^2(0, T; L^2(\Gamma ))$ for all $T>0$,  and a subsequence, that we still denote $f_n$, satisfying
\begin{eqnarray}
&&\lim _{ n\to \infty }||f_n-f|| _{  L^2(0, T; L^2(\Gamma ))}=0,\,\,\,\forall T>0,\label{L2limit}\\
&&\lim _{ n\to \infty }f_n(t, k)=f(t, k),\,\,\,a.e.\,\, t>0, k>0.\label{AElimit}
\end{eqnarray}
On the other hand, it also follows from (\ref{S3Ek22}), (\ref{S3Ek24}) and  (\ref{S3Ek25}) that $(f_n) _{ n\in \NN }$ is now a Cauchy sequence in $C([0, T); L^2(\RR_+))$. We then deduce that, for all $T>0$:
\begin{eqnarray}
&&f\in L^\infty((0, T); L^2(\RR_+))\cap C([0, T); L^2(\RR_+)), \label{regf}\\
&&\lim _{ n\to \infty }||f_n-f|| _{ L^\infty(0,T; L^2(\RR_+)) }=0.\label{limf}
\end{eqnarray}

We now take the limit in  (\ref{S3E10}) as $n\to \infty$  to obtain:
$$
||f(T)||_2^2+2C_*\int_0^T||f(t)-\mathbb{P}f(t)||^2 _{ L^2(\Gamma ) }dt  \le  \|f_{0}\|_2^2,\,\,\,\forall T>0,
$$
and then,
\begin{equation}
\label{S3Ealmost3}
||f(t)||_2^2+2C_*\int_0^\infty||f(t)-\mathbb{P}f(t)||^2 _{ L^2(\Gamma ) }dt  \le  2\|f_{0}\|_2^2
\end{equation}

Let us show now that   $\partial _t f \in L^2(0, T; L^2(\Gamma ^{-1}))$  and $f$ satisfies the equation  (\ref{E8bis}) in $L^2(0, T;  L^2(\Gamma ^{-1}))$, for all $T>0$. To this end we notice that for all $u\in L^2(0, T; L^2(\Gamma ))$ and $v\in L^2(0, T; L^2(\Gamma ))$:
$$
\left|\int _0^ T\int _0^\infty E(u)(s, k)v(s, k)dkds\right|\le (1+2 C_0)||u|| _{L^2(0, T;  L^2(\Gamma ) )}||v|| _{L^2(0, T;  L^2(\Gamma ) )}
$$
Then, the linear operator:
$$
\mathcal T: \,\,\,v\to \int _0^ T\int _0^\infty E(u)(s, k)v(s, k)dkds
$$
is linear and bounded from $L^2(0, T; L^2(\Gamma ))$ to $\RR$. It belongs then  to \hfill\break $(L^2(0, T; L^2(\Gamma )))'$. We deduce the existence of $\omega \in L^2(0, T; L^2(\Gamma ))$ such that, for all $v\in L^2(0, T; L^2(\Gamma ))$:
$$
T(v)=\int _0^ T\int _0^\infty E(u)(t, k)v(t, k)dkdt=\int _0^ T\int _0^\infty \omega (t, k)v(t, k)\Gamma (k)dkdt.
$$
Then,
 $$
 E(u)(t, k)=\omega(t, k) \Gamma (k),\,\,\hbox{for}\,\,\,a.e. \,\,t\in (0, T),\,\,\hbox{and} \,\,\,a.e.\,\, k>0.
 $$
This implies that $E(u)\in L^2(0, T; L^2(\Gamma ^{-1}))$ and we have:
\bear
\label{S3EW24}
||E(u)|| _{ L^2(0, T; L^2(\Gamma ^{-1}))}\le (1+2C_0)||u|| _{L^2(0, T;  L^2(\Gamma ) )}.
\eear
On the other hand, we know by (\ref{propsol3}) that $f-\mathbb{P}(f)\in L^2(0, \infty; L^2(\Gamma ))$. But we also have  
$\mathbb{P}(f)(t)\in L^\infty(0, \infty; L^2(\Gamma ))$ since , for all $t>0$:
\begin{eqnarray*}
||\mathbb{P}(f)(t)|| _{ L^2(\Gamma ) }=|\langle f(t), \varphi _0\rangle|||\varphi _0|| _{ L^2(\Gamma ) }\le ||f_0||_2||\varphi _0|| _{ L^2(\Gamma ) }
\end{eqnarray*}
we deduce, that $f\in L^2(0, T; L^2(\Gamma))$, then $E(f)\in L^2(0, T; L^2(\Gamma ^{-1}))$ and by \eqref{L2limit}, for a new subsequence still denoted $(f_n)$:
\begin{eqnarray}
||E(f_n)-E(f)|| _{ L^2(0, T; L^2(\Gamma ^{-1}))}\le (1+2C_0)||f_n-f|| _{L^2(0, T;  L^2(\Gamma ) )}\to 0\label{L2convE}
\end{eqnarray}
as $n\to \infty$ and
\begin{eqnarray}
\lim _{ n\to \infty } E(f_n)(t, k)=E(f)(t, k),\,\, a.e.\,\, t\in (0,T), k>0. \label{AEconvE}
\end{eqnarray}
We then deduce, passing to the limit in (\ref{nequation}),  that  $\partial _t f \in L^2(0, T; L^2(\Gamma ^{-1}))$ and  $f$ satisfies the equation  (\ref{E8bis}) in $L^2(0, T;  L^2(\Gamma ^{-1}))$, for all $T>0$. Moreover, by (\ref{S3EW24}):
\bear
\label{propsol3253We}
\left|\left|  \frac {\partial f} {\partial t} \right|\right|_{ L^2((0, T), L^2(\Gamma ^{-1})) }\le 
 (1+2C_0)||f||_{ L^2((0, T), L^2(\Gamma ))  },\,\,\,\forall T>0.
\eear
We leave the proof of  (\ref{propsol3250We}) until the end of the proof of Proposition \ref{TheoremExistence}. 

In order to prove (\ref{conservation1}) we first notice that, using  $\partial _t f\in L^2(0, T;  L^2(\Gamma ^{-1}))$ and Lemma \ref{Lemma2}, we can multiply the equation (\ref{E8bis}) by any function $\varphi \in L^2(\Gamma )$ to obtain:
$$
\frac {d} {dt}\langle  f, \varphi \rangle =\langle E(f), \varphi \rangle.
$$
By Lemma \ref{Lemma20}, identity (\ref{conservation1}), and then (\ref{conservation2}) follows.

From  (\ref{conservation2}) we now deduce that, 
$$
\mathbb{P}(f)(t)=\langle f(t), \varphi _0\rangle \varphi _0=\langle f_0, \varphi _0\rangle \varphi _0=\mathbb{P}(f_0)\,\,\,\forall t>0,
$$
and by (\ref{S3Ealmost3}),  (\ref{propsol3}) immediately follows. We then easily deduce  (\ref{propsol0}), (\ref{propsol1}).

We prove now (\ref{propsol30}).  Since  $f_n$ satisfies (\ref{truncsolution2}), (\ref{nequation}) and  (\ref{ndata}), we obtain after integration on $(0, t)$:
\begin{equation}
\label{nmildequation}
f_n(t, k)-f _{ 0, n }(k)=\int _0^t E(f_n)(s, k)ds,\,\,\,\forall n>0,\,\forall t>0, \,\forall k>0.
\end{equation}
Using now (\ref{L2convE}) we notice that, for all $t>0$:
\begin{equation*}
\left\| \int _0^t \left( E(f_n)(s)-E(f)(s)\right)ds\right\| _{ L^2(\Gamma ^{-1}) }\!\!\!\!\!\!\le
C_0 \sqrt t\, \|f_n-f\| _{L^2(0, t;  L^2(\Gamma)) }.
\end{equation*}
We deduce that
$$
\lim _{ n\to 0 }\left\| \int _0^t  E(f_n)(s)ds- \int _0^t E(f)(s)ds\right\| _{ L^2(\Gamma ^{-1}) }=0
$$
and then, up to a new subsequence still denoted $(f_n)$:
\begin{equation}
\label{AEconvintE}
\lim _{ n\to 0 }\int _0^t  E(f_n)(s)ds= \int _0^t E(f)(s)ds=0,\,\,\, a.e. \,\,k>0
\end{equation}
Using now (\ref{AElimit}), (\ref{AEconvE}) and (\ref{AEconvintE}) we first pass to the limit in (\ref{nmildequation})  as $n\to \infty$ for almost every $t\in (0,T)$ and $k>0$ and deduce that:
$$
f(t, k)=f_0(k)+\int _0^tE(f)(s, k)ds,\,\,\,a.e. \,\,t\in (0,T), k>0.
$$
Therefore,
\begin{eqnarray*}
\lim _{ t\to 0 }||f(t)-f_0|| _{  L^2(0, t;  L^2(\Gamma ^{-1}))}&\le& C_0\int _0^t||f(s)|| _{ L^2(\Gamma ) }\\
&\le &C_0\sqrt t\,||f|| _{ L^2(0, t;  L^2(\Gamma )) }.
\end{eqnarray*}
Since, on the other hand, $f\in C([0, T); L^2(\RR_+))$,  (\ref{propsol30}) follows.

If we assume that $f_0\ge 0$, we have seen that, for every $n$,  $f_n(t)\ge 0$ for all $t>0$. We deduce by (\ref{AElimit}) that $f(t, k )\ge 0$ for all $t>0$ and a. e. $k>0$.

Finally, in order to prove  the estimate  we argue as follows. Consider the function $g(t, k)=f(t, k)-\mathbb{P}(f_0)$. By (\ref{conservation2}), $g$ satisfies all the properties that have been already proved for the function $f$. Moreover, by construction $\mathbb{P}(g)(t)=0$ for all $T\ge 0$. Therefore, using (\ref{propsol3253We}):
\bean
\left|\left|  \frac {\partial g} {\partial t} \right|\right|^2_{ L^2((0, T), L^2(\Gamma ^{-1})) }\le 
 (1+2C_0)^2||g||^2_{ L^2((0, T), L^2(\Gamma ))  }
 \eean
and then, 
\begin{equation}
\label{propsol3250WeT}
\left|\left|  \frac {\partial f} {\partial t} \right|\right|^2_{ L^2((0, T), L^2(\Gamma ^{-1})) }\le  (1+2C_0)^2||f(t)-\mathbb{P}(f_0)||^2_{ L^2((0, T), L^2(\Gamma ))  }, \forall T>0
\end{equation}
from where (\ref{propsol3250We}) follows.
\end{proof}

\section{Rate of decay}\label{Rate}
In this Section we prove the algebraic  rate of convergence  of the solutions obtained in Section \ref{SectionExistence} towards the corresponding equilibrium. 
To this end we first need the following Lemma.
\begin{lemma}
\label{conditiondecay}
Let $f_0\in L^2(\RR_+)$ such that $\int _0^\infty f_0(k)\varphi _0(k)dk=0$ and satisfies  (\ref{BoltzmannPhononTheoremExponentialDecayCondition}) or
(\ref{BoltzmannPhononTheoremExponentialDecayCondition2}).
Suppose  that there exist $C^*>0$, $\omega>0$ and $\tau >0$ such that, the   solution $f$ of (\ref{E8bis})--(\ref{E12b})  obtained in Proposition  \ref{TheoremExistence} satisfies:
\begin{equation}
\|f(t)\|_2\leq C^*\|f_0\|_2(t+1)^{-\omega}\,\,\,\,\forall t\ge \tau .\label{S4Cond}
\end{equation}
Then, there exist $\theta_1>0$,  $\kappa_1>0$ and $\kappa_2>0$, where $\kappa _1$ and $\kappa _2$ are independent on $\theta_1$, such that, for all $0<\theta<\theta_1$ and for all $t>\max\{1, \tau\}$
\begin{equation}\label{Assumptiona}
\int_{0}^\infty |f (t,k)|^2\Gamma  (k)dk \geq \kappa_1\theta\int_{0}^\infty |f(t,k)|^2dk-\kappa_2\left(\frac{\theta^2}{(t+1)^{2\omega}}+\frac{\theta}{(t+1)}\right).
\end{equation}
\end{lemma}
\begin{proof}
By hypothesis:
$$
\frac {\partial f} {\partial t}=-\Gamma (k)f(t, k)+\int _0^\infty K(k, k')f(t, k')dk'.
$$
Multiply both sides of the above equation by $2f$, we get
$$
\frac {\partial f^2} {\partial t}=-2\Gamma (k)f^2(t, k)+2\int _0^\infty K(k, k')f(t, k')dk'f(t,k).
$$

Using (\ref{propG1}) and (\ref{kernelK1})  in the Appendix we deduce, that there exist two positive constants $\theta_0<1$ and $C_K$ such that, for all $k\in (0, \theta_0)$:
\begin{eqnarray*}
&&(i)\quad \Gamma (k)\ge \frac {k} {2},\\
&&(ii)\quad\int _0^\infty K(k, k')f(t, k')dk'\le ||f(t)|| _{ 2 }||K(k, \cdot)||_2 \le \frac {C_K} {2}\,\,\, k ||f(t)|| _{ 2 }.
\end{eqnarray*}
Therefore, for $\theta \in (0, \theta_0)$ and all $t>0$:
$$
\frac {\partial f^2} {\partial t}(t, k)\le - {k} f^2(t, k)+C_K k ||f(t)|| _{ 2 }|f(t, k)|\,\,\,\, a. e. k\in (0, \theta).
$$
Using now (\ref{S4Cond}) we deduce, for $\theta \in (0, \theta_0)$ and all $t>\tau $:
\begin{eqnarray*}
&&\frac {\partial f^2} {\partial t}(t, k)+ {k} f^2(t, k)\le C_K k ||f(t)|| _{ 2 }|f(t, k)|\leq C_KC^*k(t+1)^{-\omega}|f(t, k)|||f_0|| _{ 2 }\\
&&\frac {\partial } {\partial t}\left(f^2(t, k)e^{ {k} t} \right)\le C_KC^*k(t+1)^{-\omega}e^{{k}t}|f(t, k)|||f_0|| _{ 2 }.
\end{eqnarray*}
Since
\begin{eqnarray*}
\frac {\partial } {\partial t}\left(f^2(t, k)e^{ {k} t} \right)&=&\frac {\partial } {\partial t}\left(\left(f(t, k)e^{\frac {k} {2}t} \right)^2\right)\\
&=&2\left|f(t, k)e^{\frac {k} {2}t} \right|\frac {\partial } {\partial t}\left|f(t, k)e^{ \frac {k} {2}t} \right|\mbox{ for a. e. } k,
\end{eqnarray*}
then
\begin{eqnarray*}
&&\frac {\partial } {\partial t}\left(|f(t, k)|e^{\frac {k} {2}t} \right)\le \frac{C_KC^*}{2}k||f_0|| _{ 2 }(t+1)^{-\omega}e^{\frac {k} {2}t}\\
&&|f(t, k)|e^{\frac {k} {2}t}\le |f_0(k)|+ \frac{C_KC^*}{2}||f_0|| _{ 2 }k\int_0^t(s+1)^{-\omega}e^{\frac {k} {2}s}ds.
\end{eqnarray*}
By lemma \ref{Hopital} with $\rho=k/2$ and $\theta= \omega$:
\begin{eqnarray*}
\int_0^{t}({s}+1)^{-\omega}e^{\frac {k{s}} {2}}d{s}&\leq &C _{ \omega  }[(t+1)^{-\omega }+e^{-\frac {kt} {6}}]\frac{e^{\frac {k} {2}  t}}{k }
\end{eqnarray*}
for all $\omega >0$,  and $t>0$, where we can take $C _{ \omega  }=6\times 2^{\omega }$.
Then, for all $t>\tau$ and $\theta\in (0, \theta_0)$:
\begin{eqnarray*}
&& |f(t, k)|e^{\frac {k} {2}t}\le |f_0(k)|+ \frac{C_KC^*C_{\omega }}{2}||f_0|| _{ 2 }\left[(t+1)^{-\omega}+e^{-\frac {kt} {6}}\right]e^{\frac {k} {2}t}\\
&&  |f(t, k)|\le |f_0(k)|e^{-\frac {kt} {2}}+ \frac{C_KC^*C_{\omega }}{2}||f_0|| _{ 2 }\left[(t+1)^{-\omega}+e^{-\frac {kt} {6}}\right]\\
&& |f(t, k)|^2\le 2|f_0(k)|^2e^{-{k}t}+A||f_0|| _{ 2 }^2\left[(t+1)^{-2\omega}+e^{-\frac {kt} {3}}\right]\\
&&A=(C_KC^*C_{\omega })^2.
\end{eqnarray*}
As a consequence, if $0<\theta \le \theta_0$:
\begin{eqnarray}
\label{S4E12a}
&&\int_{0}^\theta|f(t, k)|^2dk  \leq 2\int_{0}^\theta f_0^2(k)e^{-kt}+A||f_0|| _{ 2 }^2\left(\frac{\theta}{(1+t)^{2\omega }}+
\frac {3} {t}\right).
\end{eqnarray}
If we now assume that $f_0$ satisfies (\ref{BoltzmannPhononTheoremExponentialDecayCondition}):
$$I=\int_{0}^1\frac{|f_0(k)|^2}{k}dk<\infty,$$
then we obtain, for all $t\ge \max\{1, \tau \}$:
\begin{equation}
\label{S4E1a}
\int_{0}^\theta|f(t, k)|^2dk \leq \frac{2I}{(t+1)}+A||f_0|| _{ 2 }^2\left[\frac{\theta}{(t+1)^{2\omega}}
+\frac {3} {1+t}\right].
\end{equation}
On the other hand, by (\ref{propG1}) and (\ref{propG2}) it easily follows that there exists a positive constant $\kappa >0$ such that for all $k>0$ we have $\Gamma (k)\ge \kappa\, k$.  We then have:
\begin{eqnarray*}
&&\int_{0}^\infty |f(t, k)|^2\Gamma(k)dk =\int_{0}^\theta |f(t, k)|^2\Gamma (k)dk + \int_{\theta}^\infty |f(t, k)|^2\Gamma(k)dk\\
&\geq &\kappa\,\theta\int_{\theta}^\infty |f(t, k)|^2dk\\
&=&-\kappa\,\theta\int_{0}^\theta |f(t, k)|^2dk+\kappa\,\theta\int_{0}^\infty |f(t, k)|^2dk\\
&\geq& -\kappa \theta\left(\frac{2I}{(t+1)}+(C_KC^*C_{\omega }||f_0|| _{ 2 })^2\left[\frac{\theta}{(t+1)^{2\omega}}
+\frac {3} {1+t}\right]\right)+\\
& &+ \kappa \theta\int_{0}^\infty |f(t, k)|^2dk.
\end{eqnarray*}
Then, condition (\ref{Assumptiona}) is satisfied with
\begin{eqnarray}
&&\kappa_1=\kappa , \label{constants6}\\
&&\kappa_2=\kappa \left(2I+4A||f_0|| _{ 2 }^2\right) \label{constants7}.
\end{eqnarray}
for all $t\ge \max\{1, \tau \}$.

If, on the other hand, the initial data $f_0$ satisfies (\ref{BoltzmannPhononTheoremExponentialDecayCondition2})  then, by Lebesgue convergence Theorem: 
$$
\lim _{ t\to \infty }t\int_{0}^\theta f_0^2(k)e^{-kt}dk=\lim _{ t\to \infty }\int_{0}^{\theta t} f_0^2\left(\frac {x} {t}\right)e^{-x}dx=a^2
$$
Notice that if the limit $a$ exists, then the function $f_0$ is bounded in a  neighborhood   of the origin, from where, for all  $x\in (0, t\theta)$, $x/t \in (0, \theta)$ and $f(x/t)$ is bounded if $\theta_0$ is sufficiently small. We then deduce by (\ref{S4E12a}) that
\begin{equation}
\label{S4E1aBIS}
\int_{0}^\theta|f(t, k)|^2dk \leq \frac{2a^2}{(t+1)}+A||f_0|| _{ 2 }^2\left[\frac{\theta}{(t+1)^{2\omega}}
+\frac {3} {1+t}\right].
\end{equation}
Arguing as above we deduce that condition  (\ref{Assumptiona})   is now satisfied with
\begin{eqnarray}
&&\kappa_1=\kappa , \label{constants6bis}\\
&&\kappa_2=\kappa \left(2a^2+4A||f_0|| _{ 2 }^2\right)\label{constants7bis}.
\end{eqnarray}
\end{proof}
\begin{remark}
\label{remarkconstants1}
The constants $\theta_0$ and  $C_K$ are determined  by the behavior of $\Gamma (k)$ and $||K(k, \cdot)||_2$ respectively as $k\to 0$.  The value of  $\kappa$ is determined by the global behavior of the function $\Gamma $.  The constants $\kappa_1$ and $\kappa_2$ given by 
 (\ref{constants6}) and  (\ref{constants7}) or  (\ref{constants6bis}) and  (\ref{constants7bis}) depend on the  global behavior of the function $\Gamma $, but also on the quantities $\int_{0}^1\frac{|f_0(k)|^2}{k}dk$ or $a$ respectively. 
\end{remark}

The algebraic convergence rate of the solution of problem  (\ref{E8bis})--(\ref{E12b}) follows as a consequence of  Lemma \ref{conditiondecay}, using the following result.

\begin{lemma}
\label{TheoremExponentialDecay}
 Suppose that $f_0\in L^2(\RR_+)$ is  such that 
$\mathbb{P}(f_0)=0$ and satisfies (\ref{BoltzmannPhononTheoremExponentialDecayCondition}) or
(\ref{BoltzmannPhononTheoremExponentialDecayCondition2}). Then, there exists a positive constant $C$, that does not depend on $\|f_0\|_2$ such that for all $t>0$: 
\begin{eqnarray}
\label{ExDecaya1}
&&\|f(t)\|_{2}\leq  C\|f_0\|_2(1+t)^{-1/2}.
\end{eqnarray}
\end{lemma}

\begin{proof}
Since equation (\ref{E8bis}) is linear, we may suppose without any loss of generality that $||f_0||_2=1$. We divide the proof into two steps.

\textit{Step 1}. We first apply Lemma \ref{conditiondecay} with $\omega=0$. 
To this end we multiply the equation (\ref{E8bis}) by $f$ and integrate over $\RR_+$ and obtain, using Lemma \ref{LemmaCoercivity}:
 $$\frac{d }{dt}\|f\|_{2}^2=\langle E(f), f\rangle\leq -C_*\int_0^\infty| \sqrt \Gamma (k)f(k)|^2dk.$$
 Since the solution that we have obtained is such that $||f(t)||_2\le ||f_0||_2$ for all $t>0$, condition (\ref {S4Cond}) holds  with $\omega =0$, $\tau =0$ and  $C^*=1$. 
Then, by  Lemma \ref{conditiondecay},  there exist three positive constants $\theta_0$, $\kappa_1$ and $\kappa_2$, with $\kappa_1$ and $\kappa _2$ independent of $\theta_0$, such that for all $\theta \in (0, \theta_0)$ and for all $t>1 $:
$$\frac{d }{dt}\|f\|_{2}^2\leq-C_*\kappa_1\theta\|f\|_{2}^2+C_*\kappa_2\left(\theta^2+\frac{\theta}{(t+1)}\right).$$
This leads to
\begin{eqnarray}
&&\frac{d}{dt} \left(\|f\|_{2}^2\exp(C_1\theta t)\right)\leq C_2\left(\theta^2+\frac{\theta}{(t+1)}\right)\exp(C_1\theta t), \label{constants1}\\
&&\hbox{with:}\,\,\, C_1=\max\{1, C_*\kappa_1\},\,\,\,C_2=C_*\kappa_2. \label{constants2}
\end{eqnarray}
Thus, for all $t>1$:
\begin{equation*}
\|f(t)\|_{2}^2\le \exp(-C_1\theta t )+ C_2\int_0 ^t\left(\theta^2+\frac{\theta}{(s+1)}\right)\exp(-C_1\theta (t-s))ds.
\end{equation*}
and, by (\ref{Hopital0}) in Lemma \ref{Hopital}:
\begin{eqnarray}
\|f(t)\|_{2}^2
\leq\exp(-C_1\theta t )+ {C_2}\theta^2 t+ C_2
\left[\frac {2} {1+t}+3e^{-\frac {C_1\theta t} {3}} \right]\label{Step1e1}
\end{eqnarray}
for all $\theta\in (0, \theta_0)$ and $t\ge 1$.

We fix now a constant $\delta $ such that
\begin{eqnarray}
\frac {2} {3}<\delta<1,\label{delta}
\end{eqnarray}  
and define
\begin{eqnarray}
T_0=\left(\frac {1} {C_1\theta_0}\right)^\delta \label{timeT0}.
\end{eqnarray}  
Then for all $t\ge T_0 $, 
we have $t^{-\delta}C_1^{-1}\le T_0^{-\delta}C_1^{-1}=\theta_0$. We may therefore choose   $\theta=(t+1)^{-\delta}C_1^{-1}$ in (\ref{Step1e1}) to obtain that, for all $t\ge \max \{1, T_0 \}$:
\vfill
\eject
\begin{eqnarray}
\|f(t)\|_{2}^2
&\leq& \exp(-C_1 t (1+t)^{-\delta } )+ \frac {C_1^{-2}C_2 t} {(1+t)^{2\delta} }+\nonumber \\
&&\hskip  4cm + C_2
\left[\frac {2} {1+t}+3e^{-\frac {t(1+t)^{- \delta }} {3}} \right]\nonumber \\
&\le&\exp(-C_1 t (1+t)^{-\delta } )+ C_1^{-2}C_2 (1+t)^{1-2\delta }+\nonumber \\
&&\hskip  4cm+ C_2
\left[\frac {2} {1+t}+3e^{-\frac {t(1+t)^{- \delta }} {3}} \right]\\
&\le& (1+3C_2)e^{-\frac {t(1+t)^{- \delta }} {3}} + C_1^{-2}C_2 (1+t)^{1-2\delta }+
 \frac {2C_2} {1+t}.
\label{Step1e2}
\end{eqnarray}
Since $\delta <1$,  there is a unique positive number $T_1$ such that
\begin{equation}
(1+3C_2)e^{-\frac {T_1(1+T_1)^{- \delta }} {3}}=C_1^{-2}C_2 (1+T_1)^{1-2\delta }.
\end{equation}
Then, if $t\ge T_2=\max\{1, T_0, T_1\}$, 
$$(1+3C_2)e^{-\frac {t(1+t)^{- \delta }} {3}}\le C_1^{-2}C_2 (1+t)^{1-2\delta }$$
and 
\begin{eqnarray*}
\|f(t)\|_{2}^2\le 2C_1^{-2}C_2 (1+t)^{1-2\delta }+ \frac {2C_2} {1+t}.
\end{eqnarray*}
Since $\delta \in (2/3, 1)$, if we call $\omega _0=\frac {2\delta -1} {2}$we have $\omega _0\in (1/6, 1/2)$ and then
\begin{eqnarray}
\|f(t)\|_{2}^2\le 2C_2 (1+C_1^{-2})(1+t)^{-2\omega _0}\,\,\,\, \forall t\ge T_2.\label{S4step1}
\end{eqnarray}
\textit{Step 2.} Using the estimate  (\ref{S4step1}) we may apply now  Lemma \ref{conditiondecay} with $\omega=\omega _0$, $\tau=T_2$ and $ 2C_2 (1+C_1^{-2})$ in the role of $C^*$. Let us call  $2C_2 (1+C_1^{-2})=C^{**}$. Arguing as above  we first write that, by Lemma \ref{conditiondecay},  there exists three positive constants $\theta'_0$, $\kappa'_1$ and  $\kappa'_2$ with $\kappa'_1$ and $\kappa' _2$ independent of $\theta'_0$, such that for all $\theta \in (0, \theta'_0)$ and for all $t>T_2$:
$$\frac{d }{dt}\|f(t)\|_{2}^2\leq-C_*\kappa'_1\theta\|f(t)\|_{2}^2+C_*\kappa'_2\left(\frac {\theta^2} {(1+t)^{2\omega _0}}+\frac{\theta}{(t+1)}\right).$$
Then, for all $t\ge T_2$:
$$\|f(t)\|_{2}^2\le \|f_0\|_{2}^2 e^{-C_1'\theta t}+ C_2'\int_0 ^t\left(\frac{\theta^2}{(s+1)^{2\omega_0}}+\frac{\theta}{(s+1)}\right)e^{-C_1'\theta (t-s)}ds.$$
where
$$C_1'=C_*C^{**}\kappa'_1,\,\,\,C_2'=C_*\kappa'_2.$$
Using (\ref{Hopital0}):
\begin{eqnarray*}
&&\int_0 ^t\left(\frac{\theta^2}{(s+1)^{2\omega_0}}+\frac{\theta}{(s+1)}\right)e^{C_1'\theta s}ds\le \\
&&\hskip 2cm \le \theta\left(4^{\omega _0}(t+1)^{-2\omega _0}+3e^{-C'_1\theta t/3}\right)\frac{e^{C'_1\theta  t}}{C'_1 }+\\
&&\hskip 4cm +\left(2(t+1)^{-1}+3e^{-C'_1\theta t/3}\right)\frac{e^{C'_1\theta  t}}{C'_1}
\end{eqnarray*}
from where we deduce that for all $\theta\in (0, \theta_0')$ and $t\ge T_2$:
\begin{eqnarray}
&&||f(t)||_2^2\le  e^{-C_1'\theta t}+\\
&&+\frac{C_2'}{C_1'}\left( \frac{4^{\omega _0}\theta}{(t+1)^{2\omega _0}}+3\theta e^{-C'_1\theta t/3}+\frac{2}{(t+1)}+3e^{-C'_1\theta t/3} \right)\nonumber\\
&\le &\left(1+6\frac{C_2'}{C_1'} \right)e^{-C'_1\theta t/3}+
\frac{C_2'}{C_1'} \frac{2^{2\omega _0}\theta}{(t+1)^{2\omega _0}}+\frac{C_2'}{C_1'}\frac{2}{(t+1)}. \label{S4step4}
\end{eqnarray}
(where we have used that $||f_0||_2\le 1$). We define now 
\begin{eqnarray}
T_3=\left(\frac {1} {C'_1\theta'_0}\right)^\delta. \label{timeT3}
\end{eqnarray}  
 Then, if $t>\max\{T_2, T_3\}$,  $t^{-\delta}{C'_1}^{-1}\le T_3^{-\delta}{C'_1}^{-1}=\theta'_0$. We may therefore choose   $\theta=(t+1)^{-\delta}{C'_1}^{-1}$ in (\ref{S4step4}) and obtain
$$
||f(t)||_2^2\le \left(1+6\frac{C_2'}{C_1'} \right)e^{-\frac {t(1+t)^{-\delta }} {3}}+
\frac{C_2'}{{C_1'}^2} \frac{4^{\omega _0}}{(t+1)^{2\omega _0+\delta }}+\frac{C_2'}{C_1'}\frac{2}{(t+1)}
$$
for all $t\ge \max\{T_2, T_3\}$. We now call $T_4$ the positive number such that
$$
\left(1+6\frac{C_2'}{C_1'} \right)e^{-\frac {T_4(1+T_4)^{-\delta }} {3}}=
\frac{C_2'}{{C_1'}^2} \frac{4^{\omega _0}}{(T_4+1)^{2\omega _0+\delta }}
$$
then, for all $t\ge \max\{T_2, T_3, T_4\}$,
$$
||f(t)||_2^2\le 2
\frac{C_2'}{{C_1'}^2} \frac{4^{\omega _0}}{(t+1)^{2\omega _0+\delta }}+\frac{C_2'}{C_1'}\frac{2}{(t+1)}.
$$
Since $\delta >2/3$ and $2\omega _0>1/3$, $2\omega _0+\delta >1$ and for all $t\ge \max\{T_2, T_3, T_4\}$:
$$
||f(t)||_2^2\le 
2\left(\frac{C_2'4^{\omega _0}}{{C_1'}^2}+\frac{C_2'}{C_1'}\right)\frac{1}{(t+1)}.
$$
Since on the other hand, $||f(t)||_2^2\le ||f_0||_2^2=1$ for all $t\ge 0$ we deduce (\ref{ExDecaya1}) for some positive constant $C$  and for all $t>0$. If the initial data is such that $||f_0||_2\ge 1$, we apply the previous argument to the function $f(t)/||f_0||_2$ and (\ref{ExDecaya1}) by the linearity of the equation (\ref{E8bis})--(\ref{E12}).
\end{proof}

We may state now the following Corollary that follows  from Lemma \ref{TheoremExponentialDecay} and Lemma \ref{conditiondecay}.
\begin{corollary}
\label{corollary}
For any solution $f$ of  (\ref{E8bis})--(\ref{E12b}) given by Proposition  \ref{TheoremExistence} such that the initial data $f_0$ satisfies (\ref{BoltzmannPhononTheoremExponentialDecayCondition}) or (\ref{BoltzmannPhononTheoremExponentialDecayCondition2}), there exists a positive constant $C$, depending the behavior of $\Gamma (k)$ on $[0, \infty)$, of $||K(k, \cdot)||_2$  as $k\to 0$  and  on $\int_{0}^1\frac{|f_0(k)|^2}{k}dk$ or $a$ respectively,  such that, for all $t>0$:
\begin{equation}
\label{S4Ecorollary}
||f(t)-c_0\varphi _0||_2\le C\frac { ||f_0-\mathbb{P}(f_0)||_2} {(1+t)^{1/2}}.
\end{equation}
\end{corollary}
\begin{proof}
If $c_0=\int _0^\infty f_0(k)\varphi _0(k)dk=0$, the conclusion follows from Lemma \ref{conditiondecay}. 
Suppose that $c_0\not = 0$. Consider then the initial data
$$
g_0=f_0-\mathbb{P}(f_0).
$$
By the properties of $\varphi _0$ and the hypothesis on $f_0$, it easily follows that $g_0$ satisfies all the hypothesis of Lemma \ref{TheoremExponentialDecay} and Lemma \ref{conditiondecay}. The solution $g$ of the problem  (\ref{E8bis})--(\ref{E12b}) with initial data $g_0$ satisfies then
\begin{equation}
\label{estimateg}
\|g(t)\|_{2}\leq  C(1+t)^{-\frac{1}{2}}||g_0||_2.
\end{equation}
Notice on the other hand that the function
$$
G(t, k)=f(t, k)-\mathbb{P}(f_0)
$$
is also a solution of (\ref{E8bis})--(\ref{E12b}) with initial data $g_0$ satisfying properties (\ref{propsol0})--(\ref{propsol2}).
Then, by the uniqueness of solution to  (\ref{E8bis})--(\ref{E12b}) proved in Proposition \ref{TheoremExistence}, $g=f-\mathbb{P}(f_0)$ and 
(\ref{S4Ecorollary}) follows from (\ref{estimateg}).
\end{proof}

\noindent
\begin{proof}\textbf{of Theorem \ref{BoltzmannPhononTheoremExponentialDecay}}. The point (i) follows from Proposition \ref{TheoremExistence}. The point (ii) follows from Corollary \ref{corollary}.
\end{proof}

We do not know if the rate of convergence obtained in  Theorem \ref{BoltzmannPhononTheoremExponentialDecay} is optimal.
One may also wonder whether it is necessary to impose one of the conditions (\ref{BoltzmannPhononTheoremExponentialDecayCondition}), (\ref{BoltzmannPhononTheoremExponentialDecayCondition2})  in order to have the algebraic decay (\ref{BoltzmannPhononTheoremExponentialDecayExDecay}). We do not know neither if these  conditions are optimal in any sense. But we show in the next Lemma that it is not possible to have any convergence rate uniform  for all the functions in $L^2(\RR_+)\cap L^2(\Gamma )$, without any other restriction. More precisely, we have the following.

\begin{lemma}
\label{noexpdecay} There is no  function $\rho (t)\ge 0$ satisfying  $\overline{\lim} _{ t\to \infty }\rho (t)<1$ and  such that, 
for  all data $f_0\in L^2(\RR_+)\cap L^2(\Gamma )$,  the solution of (\ref{E8bis})--(\ref{E12b}) given by Proposition \ref{TheoremExistence} satisfies:
\begin{equation}\label{ExDecay}
\|f(t)-\mathbb{P}(f_0)\|_{2}\leq \rho (t)\|f_0-\mathbb{P}(f_0)\|_{2},\,\,\,\forall t>0.
\end{equation}
\end{lemma}
\begin{proof} Suppose by contradiction that such a function $\rho $ do exists. Let us call, $g(t, k)=f(t, k)-\mathbb{P}(f_0)(k)$. 
From $(\ref{ExDecay})$ we deduce that, for any $T>0$:
\begin{eqnarray}
 \|g_0\|_{2}^2- C\|g_0\|^2_{ 2 }\,\rho (T)\le \|g_0\|_{2}^2-\|g(T)\|_{2}^2=-\int_0^T\langle E(g), g\rangle dt. \label{LemmaAlphaGreaterConstant1}
\end{eqnarray}
By (\ref{LemmaAlphaGreaterConstant1}),  there exists $\delta >0$ and $T_0>0$ such that if $T>T_0$,
\begin{equation}\label{LemmaAlphaGreaterConstant10}
\delta  \|g_0\|_{2}^2\le \|g_0\|_{2}^2-\|g(T)\|_{2}^2\le -\int_0^T\langle E(g), g\rangle dt. 
\end{equation}
In order to estimate the right hand side of (\ref{LemmaAlphaGreaterConstant10}) we consider  the norm of $\|g(T)-g_0\|_2^2$:
\begin{eqnarray*}
\|g(T)-g_0\|_2^2 & = & 2\int_0^T\left<\partial_t g,g-g_0\right>dt\\
& = & \int_0^T 2\left<E (g),g-g_0\right>dt\\
&=&2\int_0^T\langle E(g), g\rangle dt-2\int_0^T\left<E(g),g_0\right>dt\\
&\le&\int_0^T\langle E(g), g\rangle dt-\int_0^T\langle E(g_0), g_0\rangle dt.
\end{eqnarray*}
where, in the last step, we have used (\ref{S2Cor2E1}) in Corollary \ref{S2Cor2}. \\
We then have:
\begin{eqnarray}
\label{S2E9}
-\int_0^T\langle E(g), g\rangle dt&\leq &-T\langle E(g_0), g_0\rangle,
\end{eqnarray}
Since $g_0\in L^2(\Gamma )$, by (\ref{Lemma2E1}):
\begin{equation}
\label{S2E10}
-\langle E(g_0), g_0\rangle \le C_0\|g_0\| _{ L^2(\Gamma  ) }^2=C_0\|\sqrt \Gamma  g_0\|_2^2.
\end{equation}
We deduce from (\ref{LemmaAlphaGreaterConstant10}), (\ref{S2E9}) and (\ref{S2E10}) that, for all $g_0\in L^2(\RR_+)\cap L^2(\Gamma)$:
\begin{equation}
\label{S4E10}
\|g_0\|_2^2\le \frac{TC_0}{\delta }\|\sqrt \Gamma  g_0\|_2^2
\end{equation}
By property (\ref{propG1}) of the function $\Gamma $ this is not possible if $g_0\in L^2(\RR_+)\cap L^2(\Gamma)$ with support in an interval $(k_1, k_2)$, with  $0<k_1<k_2$ sufficiently small.
\end{proof}

\begin{remark}\label{Remark1} The results in the Appendix say  that  
\begin{eqnarray*}
&&\Gamma (k)\sim \frac {\pi k} {15 },\,\,\, k\to 0,\\
&&||K(k, \cdot)||_2\le \frac {2\pi ^3 k} {\sqrt {21}},\,\,\, 0<k<<1.
\end{eqnarray*} 
 This suggest that
a very rough approximation of the equation (\ref{E8bis})  near $k=0$ could be given by
\begin{eqnarray*}
\label{BoltzmannPhononRemark1}
&&\frac{d}{dt}f(t, k)= -Ckf(t, k),\,\, for\,\, t>0, \,\,k\,\,small \\
&& f(0, k)=f_0(k)\,\,for\,\,k\,\,small,
\end{eqnarray*}
for some constant $C$. By the positivity of the operator $E$ it seems reasonable to have $C>0$. Since the solution $f$ of that simple equation is
$$f(t, k)=e^{-Ckt}f_0(k),\,\,\,\forall t>0,$$
we have
\begin{eqnarray*}
&&\int_0^{k_0}|f(t, k)|^2dk=\int_0^{k_0}|f_0(k)|^2 e^{-2Ckt}dk,\,\,\,\forall t>0.
\end{eqnarray*}
Therefore, if $f_0$ satisfies (\ref{BoltzmannPhononTheoremExponentialDecayCondition}),
\begin{eqnarray*}
&&\int_0^{k_0}|f(t, k)|^2dk\le\frac {1} {2Ct}\int_0^{k_0}\frac {|f_0(k)|^2} {k} dk,\,\,\,\forall t>0.
\end{eqnarray*}
If on the other hand, $f_0$ is continuous at $k=0$, 
$$t\int_0^{k_0}|f(t, k)|^2dk=\frac {1} {2C}\int _0^{2Ck_0t}\left|f_0\left(\frac {x} {2Ct}\right)\right|^2 e^{-x}dx.$$
Since, by (\ref{BoltzmannPhononTheoremExponentialDecayCondition2}),
$$
\lim _{ t\to \infty }t\int_0^{k_0}|f(t, k)|^2dk=\frac {a^2} {2C}
$$
we deduce 
$$
\int_0^{k_0}|f(t, k)|^2dk=\frac {a^2} {2C t}+o\left(\frac {1} {t}  \right),\,\,\hbox{as}\,\,t\to \infty.
$$
The convergence rate $(\ref{BoltzmannPhononTheoremExponentialDecayExDecay})$ seems then in some sense optimal.
\end{remark}
\section{Proofs of  Proposition \ref{Prop1} and Theorem \ref{theorem}.}
\label{nonradial}
We give in this Section the proofs of Proposition \ref{Prop1} and  of Theorem \ref{theorem}. These follow easily from the results that have been proved in Sections  \ref{SectionOperatorE},  \ref{SectionExistence} and  \ref{Rate}. We begin with the proof of the Proposition.\\ 

\begin{proof}\textbf{ of Proposition \ref{Prop1}.}
Point  (i) follows immediately from  the orthogonality property of the spherical harmonic functions and the fact that 
$ |p| \in L^2\left(\RR^+, \frac {dp} {\sinh ^{2}(k)} \right)$. 
In order to prove point (ii)  let us notice first of all that, if $f(k)$ is such that $f\in L^2(\RR^+)$, respectively  $f\in L^2(\Gamma )$, and we consider the function $g$ defined by the change of variables (\ref{E8}):
$$
g(p)\equiv g(|p|)=\frac{\sinh (k)} {k} f(k),\,\,\, k=\frac {c |p|} {2k_BT}
$$
then $g\in L^2\left(\RR^+, \frac {k^2} {\sinh^2(k)}dr\right)$, respectively  $g\in L^2\left(\RR^+, \frac {k^2\Gamma (k)} {\sinh^2(k)}dr\right)$. 
Moreover, by definition
$$
L(g)(|p|)=(k\, \sinh k)\,  E(f)(k),
$$
where $L$ is defined in (\ref{E5000-2}).  Then, if $f\in L^2(\Gamma )$, we have $E(f)\in L^2(\Gamma ^{-1})$ by Lemma (\ref{Lemma2}), and therefore 
$L(g)\in L^2\left(\RR^+, \frac {k^2 \,\sinh ^{2}(k)dr} {\Gamma (k)} \right)$. \\
We then deduce  that
$L(|p|)\in L^2\left(\RR^+, \frac {k^2 \,\sinh ^{2}(k)dr} {\Gamma (k)} \right)$ and therefore
$$\Lambda(p)=-M(p)\Theta (p)+\int  _{ \RR^3 } \Theta (p')\,W(p, p')dp'\in L^2\left(\frac {\,\sinh ^{2}(k)dp} {\Gamma (k)} \right)$$
It is then enough to check that all the components $\Lambda _{ \ell\, m }$ of the function $\Lambda$ in the spherical harmonic basis are zero. Using the orthonormality properties of the spherical harmonic functions $Y _{ \ell\, m }$ and the definitions of the Legendre's polynomial we readily check  that these components are, up to a constant factor:
$$
\Lambda _{ \ell\, m }(|p|)=-M(|p|)\Theta _{ \ell\, m }(|p|)+ \frac {1} {2\ell +1}\int  _{ 0 }^\infty \Theta _{ \ell\, m }(r')W _{ \ell }(|p|, r') dr'
$$

Since, by Corollary \ref{coro}, the function $\phi (k)$ satisfies $E(\phi )=0$ and the function 
$\Theta_{ \ell\, m }(r)=c _{ \ell\, m }r$ is obtained from $\phi (k)$ through the change of variables (\ref{E8}), it follows that $\Lambda _{ \ell\, m }(|p|)=0$ for all $\ell$ and $m$.
\end{proof}

\begin{proof}\textbf{ of Theorem \ref{theorem}.}
We decompose  the initial data $\Omega _0$ that by hypothesis belongs to $ L^2\left(\RR^3, \frac {dp} {\sinh ^{2}(k)} \right)$ using the basis of $L^2(\Sx^2)$ of spherical harmonics:
\bean
\Omega _0(p)=\sum _{ \ell=0 }^\infty \sum _{ m=-\ell }^\ell\Omega  _{ 0, \ell\, m }(|p|)Y _{ \ell\, m }\left(\frac {p} {|p|} \right).
\eean
Using the orthonormality of the basis $\{Y _{ \ell\, m } \} $ we deduce
\bean
||\Omega_0||^2 _{ L^2\left(\RR^3, \frac {dp} {\sinh ^{2}(k)} \right) }&=&\int  _{ \RR^3 }\left| \sum _{ \ell=0 }^\infty \sum _{ m=-\ell }^\ell\Omega  _{ 0, \ell\, m }(|p|)Y _{ \ell\, m }\left(\frac {p} {|p|} \right)\right|^2 \frac {dp} {\sinh ^{2}(k)}\\
&=&\int  _{ \Sx^2} d\sigma \int _0^\infty\left| \sum _{ \ell=0 }^\infty \sum _{ m=-\ell }^\ell\Omega  _{ 0, \ell\, m }(|p|)Y _{ \ell\, m }\left(\sigma  \right)\right|^2\frac {|p|^2d|p|} {\sinh ^{2}(k)}\\
&=& \sum _{ \ell=0 }^\infty \sum _{ m=-\ell }^\ell \int _0^\infty |\Omega  _{ 0, \ell\, m }(|p|)|^2\frac {|p|^2d|p|} {\sinh ^{2}(k)},
\eean
and then:
\bean
\Omega  _{ \ell\, m }\in L^2\left(\RR^+; \frac {|p|^2d|p|} {\sinh ^{2}(k)}\right),\,\,\,\forall \ell \in \NN, m\in \{-\ell, -\ell+1, \cdots, \ell-1, \ell\}
\eean
Therefore, if we define:
\bear
f _{ 0, \ell, m }(k)=k \frac{\Omega _{0, \ell\, m}(|p|)}{\sinh k}, \,\,\,k=\frac{c|p|}{2k_BT}  \label{E8initial}
\eear
it follows that $f _{ 0, \ell, m }\in L^2(\RR^+)$. Let then be $f _{ \ell, m }$ the solution of the equation  (\ref{E8bis}) with initial data $f _{ 0, \ell\, m }$ given by Theorem \ref{BoltzmannPhononTheoremExponentialDecay} and define:
\bear
\label{S5E127}
\Omega _{\ell\, m}(t, r)=f _{\ell, m }(t, k )\frac{\sinh k}{k },   \,\,\,\,\, k =\frac{cr}{2k_BT}.  \label{E8initial}
\eear
It follows  from (\ref{propsol3000}) that:
\bear
||\Omega _{\ell\, m}(t)||^2 _{ L^2\left(\mathbb{R}_+; \frac {r^2} {\sinh^2 k }\right) }\le 2||\Omega _{0, \ell\, m}||^2 _{ L^2\left(\mathbb{R}_+; \frac {r^2} {\sinh^2 k }\right) }\,\,\,\forall t>0.\label{S5EW0}
\eear
We deduce that
\bear
\sum _{ \ell=0 }^\infty \sum _{ m=-\ell }^\ell
||\Omega _{\ell\, m}(t)||^2 _{ L^2\left(\mathbb{R}_+; \frac {r^2} {\sinh^2 k }\right) }&\le & 2\sum _{ \ell=0 }^\infty \sum _{ m=-\ell }^\ell
||\Omega _{0, \ell\, m}||^2 _{ L^2\left(\mathbb{R}_+; \frac {r^2} {\sinh^2 k }\right) }\nonumber\\
&=& 2\,||\Omega _0||^2 _{  L^2\left(\RR^3, \frac {dp} {\sinh ^{2} k} \right) }\label{S5EW1}
\eear
and  the following  function is then well defined in  $L^2\left(\RR^3, \frac {dp} {\sinh ^{2}k} \right)$ for all $t>0$:
\bean
\Omega (t, p)=\sum _{ \ell=0 }^\infty \sum _{ m=-\ell }^\ell \Omega  _{ \ell\, m }(|p|)Y _{ \ell\, m }\left(\frac {p} {|p|} \right). 
\eean
It follows from (\ref{S5EW0}), (\ref{S5EW1}) and (\ref{propsol1}) that $\Omega $ satisfies (\ref{S1E250}).

Similarly, by (\ref{propsol3000}) and (\ref{propsol3250}):  
\bean
&&\left|\left|\frac {\partial f _{ \ell\, m }} {\partial t}\right|\right|_{L^2(0, \infty; L^2(\Gamma ^{-1}(k)dk))}\le \frac{1+2C_0}{\sqrt {C_*}}\left|\left| f _{0, \ell\, m }\right|\right|_{L^2}
\eean
and then
\bean
&&\left|\left|\frac {\partial \Omega  _{ \ell\, m }} {\partial t}\right|\right|_{L^2\left(0, \infty; L^2\left(\frac {r^2} {\Gamma(k ) \, \sinh^2 k }\right)\right)}\le \frac{1+2C_0}{\sqrt {C_*}}||\Omega _{0, \ell\, m}||^2 _{ L^2\left(\mathbb{R}_+; \frac {r^2} {\sinh^2 k}\right) }
\eean

Using that $M(p)\equiv M(r) =\Gamma (k)n_0(p)(1+n_0(p))$ and $n_0(p)(1+n_0(p))=1/(4\sinh^2 k)$ we have:
\bean
\left|\left|\frac {\partial \Omega  _{ \ell\, m }} {\partial t}\right|\right|_{L^2\left(0, \infty; L^2\left(\frac {r^2 } {M(r) \sinh^4k\, }\right)\right)}\le \frac{1+2C_0}{\sqrt {C_*}}||\Omega _{0, \ell\, m}||^2 _{ L^2\left(\mathbb{R}_+; \frac {r^2} {\sinh^2 k}\right) },
\eean
and
\bean
&&\sum _{ \ell=0 }^\infty\sum _{ m=-\ell }^\ell \left|\left|\frac {\partial \Omega  _{ \ell\, m }} {\partial t}\right|\right|^2_{L^2\left(0, \infty; L^2\left(\frac {r^2 } {M(r) \sinh^4k\, }\right)\right)}\le \frac{(1+2C_0)^2}{C_*}\times \\
&&\hskip 2cm \times\sum _{ \ell=0 }^\infty\sum _{ m=-\ell }^\ell ||\Omega _{0, \ell\, m}||^2 _{ L^2\left(\mathbb{R}_+; \frac {r^2} {\sinh^2 k}\right) }=\frac{(1+2C_0)^2}{C_*}||\Omega_0||^2 _{ L^2\left(\RR^3, \frac {dp} {\sinh ^{2}(k)} \right) }.
\eean
The following  function:
\bean
\sum _{ \ell=0 }^\infty \sum _{ m=-\ell }^\ell \frac{\partial \Omega  _{ \ell\, m }} {\partial t}(|p|)Y _{ \ell\, m }\left(\frac {p} {|p|} \right) 
\eean
is then well defined in  $L^2\left(\RR^3, \frac {dp} {M(|p|)\, \sinh ^{4}k} \right)$ for all $t>0$ and
\bean
\frac {\partial \Omega } {\partial t}(t, p)=\sum _{ \ell=0 }^\infty \sum _{ m=-\ell }^\ell \frac{\partial \Omega  _{ \ell\, m }} {\partial t}(|p|)Y _{ \ell\, m }\left(\frac {p} {|p|} \right).
\eean

Since $f _{ \ell\, m }(t, k)$ satisfies the equation  (\ref{E8bis})--(\ref{E12}) in $L^2((0, \infty), L^2(\Gamma ^{-1}))$, and $M(p)\equiv M(r)=\Gamma (k)n_0(p)(1+n_0(p)$, $n_0(p)(1+n_0(p))=1/(4\sinh^2 k)$,
 the function $\Omega _{ \ell\, m } $ satisfies equation (\ref{E5000-1}), (\ref{E5000-2}) in $L^2\left(\RR^+, \frac {r^2dr} {M(r)\, \sinh ^{4}k} \right)$. One easily deduces that $\Omega $ satisfies equation (\ref{E51}) in  $L^2\left(\RR^3, \frac {dp} {M(p)\, \sinh ^{4}k} \right)$.  The two properties in (\ref{S1Einitial}) are deduced  from those in (\ref{propsol30}) using similar arguments.
 
We wish  to prove the uniqueness of solutions of (\ref{E51}) in the sense of $L^2\left(0, \infty;  L^2\left(\RR^3, \frac {dp} {M(p)\, \sinh^4 k}\right)\right)$, satisfying (\ref{S1E250})--(\ref{S1E252}) and such that
\bear
\label{S5E100}
\lim _{ t\to 0 }||\Omega (t)-\Omega _0|| _{ L^2\left(\RR^3, \frac {dp} {\sinh ^{2}k} \right)}=0.
\eear
To this end we suppose that $\Omega_1 $ and $\Omega_2 $ are two such solutions and call $\widetilde \Omega =\Omega _1-\Omega _2$. It is then also a solution  of (\ref{E51}) in  $L^2\left(0, \infty;  L^2\left(\RR^3, \frac {dp} {M(p)\, \sinh^4 k}\right)\right)$, satisfying (\ref{S1E250})--(\ref{S1E252}) and (\ref{S5E100}) with $\Omega _0=0$. It then follows that the modes $\widetilde \Omega  _{ \ell\, m }$ of $\widetilde \Omega$ satisfy equation  (\ref{E5000-1})-(\ref{E5000-2}) with initial data $\widetilde \Omega  _{ \ell\, m }(0)=0$. By the uniqueness part of Theorem (\ref{BoltzmannPhononTheoremExponentialDecay}) it follows that $\widetilde \Omega _{ \ell\, m }=0$ for each $\ell$ and $m$ and then $\widetilde \Omega \equiv0$.

 Suppose now that $\Omega _0(p)$ also satisfies (\ref{S1Econdition}). Then, for every $\ell$ and $m$, the function $f _{ 0, \ell\, m }(k)$, defined in  (\ref{E8initial}), satisfies (\ref{BoltzmannPhononTheoremExponentialDecayCondition}). By Theorem \ref{BoltzmannPhononTheoremExponentialDecay} we then have:
 \bear
 \label{S5E100}
 ||f _{ \ell\, m }(t)-c _{ 0,\ell\, m }\varphi _0|| _{ 2 }\le C\frac {||f _{ 0, \ell\, m }-c _{ 0,\ell\, m }\varphi _0||_2} {(1+t)^{1/2}}
 \eear
where
 \bear
 \label{S5E101}
 c _{ 0,\ell\, m }=\int _0^\infty f _{ 0, \ell\, m }(k)\varphi _0(k)dk.
 \eear
 Therefore, using (\ref{S5E127}) we deduce
 \bean
\int _0^\infty\left|\Omega  _{ \ell\, m }(t, r)-c _{ \ell\, m }r  \right|^2\frac {r^2dr} {\sinh^2k}\le
\frac {C} {1+t}\int _0^\infty\left|\Omega  _{0, \ell\, m }(r)-c _{ \ell\, m }r  \right|^2\frac {r^2dr} {\sinh^2 k}
\eean
where
$$
c _{ \ell\, m }=\frac {c } {2k_BT ||\phi ||_2}c _{0, \ell\, m }.
$$
If we sum now with respect to $\ell$ and $m$ we obtain 
\bean
&&\hskip -1cm ||\Omega (t)-\Theta ||^2_{L^2\left(\RR^3, \frac {dp } {\sinh^{2} k}\right)  }=\sum _{ \ell=0 }^\infty\sum _{ m=-\ell }^\ell \int  _{\RR^3}\left|\Omega  _{ \ell\, m }(t, |p|)-c _{ \ell\, m }|p|  \right|^2\frac {dp} {\sinh^2 k}\\
&&\hskip 3cm \le \frac {C} {1+t}\sum _{ \ell=0 }^\infty\sum _{ m=-\ell }^\ell \int  _{\RR^3}\left|\Omega  _{0, \ell\, m }(|p|)-c _{ \ell\, m }|p|  \right|^2
\frac {dp} {\sinh^2 k}\\
&&\hskip 3cm=\frac {C} {1+t}||\Omega (0)-\Theta ||^2_{L^2\left(\RR^3, \frac {dp } {\sinh^{2} k}\right)  }.
\eean
Since 
\bean
c _{ \ell\, m }&=&\frac {c } {2k_BT ||\phi ||_2}\int _0^\infty f _{ 0, \ell\, m }(k)\varphi _0(k)dk\\
&=&\left(\frac {c} {2k_BT} \right)^4\frac {1} {||\phi ||^2_2}\int _0^\infty \frac {\Omega  _{ 0, \ell\, m }(r)} {\sinh^2k}r^2dr\\
&=&\left(\frac {c} {2k_BT} \right)^4\frac {4} {||\phi ||^2_2}\int  _{ \RR^3 }\Omega  _{ 0}(p) Y _{ \ell\, m }\left(\frac {p} {|p|} \right)
n_0(p)(1+n_0(p))dp
\eean
and $||\phi ||_2^2=\pi ^4/30$, this concludes the proof of (\ref{S1ERatedecay})-(\ref{S1Eclm}).
 \end{proof}

\begin{remark}
\label{totalmass}
The total number of particles in the physical system described by equation (\ref{E1})-(\ref{E2}) is given  by 
$$
N(t)=\int _{ \RR^3 }n(t, p)dp.
$$
The corresponding quantity in the linear approximation that we consider in this work is:
$$
M(t)=\int _{ \RR^3 }n_0(p)dp+\int  _{ \RR^3 }n_0(p)(1+n_0(p))\Omega (t, p)dp,
$$

It follows from Theorem \ref{theorem} that, if the initial data $\Omega _0\in L^2(\RR^3)$ satisfies (\ref{S1Econdition}) then:
$$
\lim _{ t\to \infty }M(t)=\int _{ \RR^3 }n_0(p)dp+\int  _{ \RR^3 }n_0(p)(1+n_0(p))\Theta (p)dp\equiv M_\infty.
$$
where $\Theta$ is defined by (\ref{S1ETeta}) (\ref{S1Eclm}). It is easy to see that $M _{ \infty }$ may be greater or smaller than  $M(0)$. If we choose the initial data $\Omega _0=\Theta+g_0$
then,
\begin{eqnarray*}
M(0)&=&\int _{ \RR^3 }n_0(p)dp+\int _{\RR^3}n_0(p)(1+n_0(p)) \Omega _0(p)dp\\
&=&\int _{ \RR^3 }n_0(p)dp+\int _{\RR^3}n_0(p)(1+n_0(p)) \Omega _0(p)dp+\int _{\RR^3}n_0(p)(1+n_0(p)) g _0(p)dp\\
&= &M _{ \infty }+\int _{\RR^3}n_0(p)(1+n_0(p)) g _0(p)dp.
\end{eqnarray*}
The sign of $M(0)-M _{ \infty }$ is then given by the sign of $\int _{\RR^3}n_0(p)(1+n_0(p)) g _0(p)dp$ and may be positive or negative.
\end{remark}

\section{Appendix}
In this Appendix we recall the definition of Legendre's polynomials, we describe the formal approximation argument leading to the simplified equation (\ref{E5000-1})-(\ref{E5000-2}) and  present  some auxiliary results on the functions $\Gamma $ and $K$ that appear in the operator $E$ defined in (\ref{E8bis}).
\subsection{The functions $\Gamma $ and $K$.}
We present in this Appendix some auxiliary results, in  particular several  properties of the functions $\Gamma $ and $K$ that are needed in the proof of our main results. They have already been obtained in  \cite{Buot:ORT:1997} and we state and prove them here just  for the sake of completeness.
\begin{lemma}
\label{Gamma}
The function $\Gamma $  defined in (\ref{E10}), (\ref{E12}) satisfies $\Gamma \in C(0, \infty)$ and $\Gamma (k)>0$ for all $k>0$. Moreover,
\begin{eqnarray}
\lim _{ k\to 0 }\frac {\Gamma (k)} {k}=\frac {\pi ^4} {15}\label{propG1}\\
\lim _{ k\to \infty }\frac {\Gamma (k)} {k^5}=\frac {1} {15}.\label{propG2}
\end{eqnarray}
\end{lemma}
\begin{proof}
The continuity of $\Gamma $ follows immediately from the integrability properties of the integrand in  (\ref{E10}). The strict positivity of $\Gamma (k)$ for $k>0$ is deduced from the fact that the integrand in  (\ref{E10}) is non negative. In order to prove (\ref{propG1}) and (\ref{propG2}) we first notice that, by a simple change of variables, the function $\Gamma $ may be written as:
\begin{equation}
\label{GammaBis}
\Gamma (k)=\sinh k\int _0^k \phi (k-k')\phi (k') dk'+2\sinh k\int _0^\infty \phi (k+k')\phi (k') dk'
\end{equation}
By Lebesgue's convergence Theorem it follows that
$$
\lim _{ k\to 0 }\int _0^k \phi (k-k')\phi (k') dk'=0
$$
and
$$
\lim _{ k\to 0 }\int _0^\infty \phi (k+k')\phi (k') dk'=\int _0^\infty \frac {x^4} {\sinh ^2x}dx=\frac {\pi ^4} {30}
$$
from where (\ref{propG1}) follows.  

On the other hand, 
\begin{eqnarray*}
\sinh k \int _0^k\phi (k-k')\phi (k')dk'= k^5\sinh k\int _0^1\frac {z^2(1-z)^2} {\sinh (k(1-z))\sinh (kz)}dz
\end{eqnarray*}
But,
\begin{eqnarray*}
\sinh k\frac {z^2(1-z)^2} {\sinh (k(1-z))\sinh (kz)}=\frac {1-e^{-2k}} {2}\frac {e^{k}z^2(1-z)^2} {\frac {e^{k-kz}-e^{kz-k}} {2}\frac {e^{kz}-e^{-kz}} {2}}\\
=\frac {2(1-e^{-2k})z^2(1-z)^2} {(e^{-kz}-e^{kz-2k})(e^{kz}-e^{-kz})}=\frac {2(1-e^{-2k})z^2(1-z)^2} {1-e^{-2kz}-e^{-2k(1-z)}+e^{-2k}}\\
=2(1-e^{-2k})\frac {z^2} {(1-e^{-2kz})}\frac {(1-z)^2} {(1-e^{-2k(1-z)})}.
\end{eqnarray*}
And we observe that, 
\begin{eqnarray*}
\frac {x^2} {1-e^{-Ax}}\le \frac{x^2}{1-e^{-1/2}}, \,\,\forall x\in \left(\frac {1} {2A}, 1\right).
\end{eqnarray*}
When $Ax\in (0, 1/2)$, $1-e^{-Ax}\ge Ax/2$ so,
$$
\frac {x^2} {1-e^{-Ax}}\le \frac {2x} {A} \,\,\forall x\in \left(0, \frac {1} {2A}\right),
$$
and then, for $A>1$:
$$
\frac {x^2} {1-e^{-Ax}}\le 2x, \,\,\forall x\in (0, 1)
$$
since  $2 (1-e^{-1/2})<1$. 
This gives
$$
\sinh k\frac {z^2(1-z)^2} {\sinh (k(1-z))\sinh (kz)}\le 8(1-e^{-2k})z(1-z)
$$
for all $z\in (0, 1)$ and $k>1$. The Lebesgues convergence Theorem gives then,
\begin{eqnarray*}
&&\lim _{ k\to \infty }k^{-5}\left(\sinh k\int _0^k \phi (k-k')\phi (k') dk'\right)=\\
&&\hskip 4cm =\lim _{ k\to \infty }\sinh k\int _0^1\frac {z^2(1-z)^2} {\sinh (k(1-z))\sinh (kz)}dz\\
&&\hskip 4cm=2\int _0^1z^2(1-z)^2dz=\frac {1} {15}.
\end{eqnarray*}
It is not difficult to check, using similar arguments, that the second integral in the right hand side of (\ref{GammaBis}) is of lower order when $k\to \infty$ and then
(\ref{propG1})  follows.
\end{proof}

\begin{lemma}
\label{kernelK}
\begin{equation}
\label{kernelK1}
\int_0^\infty |K(k,k')|^2 dk'<\frac{4}{15}\pi^4 k^4+\frac{4}{21}\pi^6k^2
\end{equation}
and
\begin{equation}
\label{kernelK2}
\int_0^\infty\int_0^\infty \left|\frac{K(k,k')}{\sqrt{\Gamma(k')\Gamma(k)}}\right|^2dkdk'<+\infty.
\end{equation}
\end{lemma}
\begin{proof}

$$
\int_0^\infty |K(k,k')|^2 dk'\le 2k^2\int_0^\infty \left(\phi (k-k')^2+\phi (k+k')^2 \right) k'^2dk'
$$
\begin{eqnarray*}
\int_0^\infty \phi (k+k')^2 k'^2dk'= \int_0^\infty \frac {(k+k')^4} {\sinh^2 (k+k')}k'^2dk'\\
=\int_k^\infty \frac {z^4} {\sinh^2 z}(z-k)^2dz\le \int_k^\infty \frac {z^4} {\sinh^2 z}z^2dz\\
\le \int_0^\infty \frac {z^4} {\sinh^2 z}z^2dz=\frac {\pi ^6} {42}
\end{eqnarray*}

\begin{eqnarray*}
\int_0^\infty \phi (k-k')^2 k'^2dk'= \int_0^\infty \frac {(k'-k)^4} {\sinh^2 (k'-k)}k'^2dk'\\
=\int_{-k}^\infty \frac {z^4} {\sinh^2 z}(z+k)^2dz\le \int_{-\infty}^\infty \frac {z^4} {\sinh^2 z}(z+k)^2dz\\
\le 2\int_0^\infty \frac {z^4} {\sinh^2 z}(z^2+k^2)dz=2\frac {\pi ^6} {42}+2k^2\frac {\pi ^4} {30}
\end{eqnarray*}

In order to prove (\ref{kernelK2}) we first write:
\begin{eqnarray*}
\int_0^\infty\int_0^\infty \left|\frac{K(k,k')}{\sqrt{\Gamma(k)\Gamma (k')}}\right|^2dkdk'=I_1+2I_2+I_4\\
I_1=\int_0^1\int_0^1\left|\frac{K(k,k')}{\sqrt{\Gamma(k)\Gamma (k')}}\right|^2dkdk'\\
I_2=\int_0^1\int_1^\infty\left|\frac{K(k,k')}{\sqrt{\Gamma(k)\Gamma (k')}}\right|^2dkdk'\\
I_3=\int_1^\infty\int_1^\infty\left|\frac{K(k,k')}{\sqrt{\Gamma(k)\Gamma (k')}}\right|^2dkdk'\\
\end{eqnarray*}
We  notice that,
\begin{eqnarray*}
I_3=\int_1^\infty\int_1^\infty\left|\frac{K(k,k')}{\sqrt{\Gamma(k)\Gamma (k')}}\right|^2dkdk'
\le 2 \left( \int_1^\infty\int_1^\infty\frac{|K(k,k')|^2}{|\Gamma(k)|^2}dkdk'\right)
\end{eqnarray*}
since:
\begin{eqnarray*}
&&\int_1^\infty\int_1^\infty\left|\frac{K(k,k')}{\sqrt{\Gamma(k)\Gamma (k')}}\right|^2dkdk'=\int_1^\infty\int_1^\infty\left|\frac{K(k,k')}{\sqrt{\Gamma(k)\Gamma (k')}}\right|^21 _{ \Gamma (k)>\Gamma (k') }dkdk'+\\
&&\hskip 5.5cm +\int_1^\infty\int_1^\infty\left|\frac{K(k,k')}{\sqrt{\Gamma(k)\Gamma (k')}}\right|^21 _{ \Gamma (k)<\Gamma (k') }dkdk'\\
&&=\int_1^\infty\int_1^\infty\left|\frac{K(k',k)}{\sqrt{\Gamma(k')\Gamma (k)}}\right|^21 _{ \Gamma (k')>\Gamma (k) }dkdk'+\\
&&\hskip 5cm
+\int_1^\infty\int_1^\infty\left|\frac{K(k,k')}{\sqrt{\Gamma(k)\Gamma (k')}}\right|^21 _{ \Gamma (k)<\Gamma (k') }dkdk'\\
&&=2\int_1^\infty\int_1^\infty\left|\frac{K(k,k')}{\sqrt{\Gamma(k)\Gamma (k')}}\right|^21 _{ \Gamma (k')>\Gamma (k) }dkdk'\\
&&\le 2\int_1^\infty\int_1^\infty\left|\frac{K(k,k')}{\Gamma(k)}\right|^21 _{ \Gamma (k')>\Gamma (k) }dkdk'
\le 2\int_1^\infty\int_1^\infty\left|\frac{K(k,k')}{\Gamma(k)}\right|^2dkdk'
\end{eqnarray*}
and this integral $I_3$ converges by Lemma \ref{Gamma} and  (\ref{kernelK1}). 
\\ We have that 
$$\lim_{k,k'\to 0}\frac{\sqrt{\Gamma(k)\Gamma (k')}}{\sqrt{kk'}}=\frac{\pi^4}{15},$$
and
$$\lim_{k,k'\to 0}\frac{K(k, k')}{\sqrt{kk'}}=0,$$
we have
$$\lim_{k,k'\to 0}\frac{K(k, k')}{\sqrt{\Gamma(k)\Gamma (k')}}=0.$$
Therefore  $\frac{K(k, k')}{\sqrt{\Gamma(k)\Gamma (k')}}$ is continuous  on $[0, 1]\times [0, 1]$  and the first integral is then convergent. Finally let us estimate $I_2$. We first notice that for all $k>0$, $\Gamma (k)>0$ and then, by the continuity of $\Gamma $ on $[0, \infty)$ and (\ref{propG2}), there exists a positive constant $C>0$ such that
$$
\Gamma (k)\ge C>0,\,\,\,\forall k\ge 1.
$$
Therefore
\begin{eqnarray*}
\int_1^\infty\left|\frac{K(k,k')}{\sqrt{\Gamma (k')}}\right|^2dk&\le &\frac {1} {C}\int_1^\infty\left|K(k,k')\right|^2dk\\
&\le & \frac {1} {C}\left(8|B_4|\pi^4 k^4+8|B_6|\pi^6k^2\right)
\end{eqnarray*}
from where we deduce, for some positive constant $C'$:
\begin{eqnarray*}
I_2=\int_0^1\int_1^\infty\left|\frac{K(k,k')}{\sqrt{\Gamma(k)\Gamma (k')}}\right|^2dkdk' 
\le C'\int_0^1\frac {k^2} {\Gamma (k)}dk
\end{eqnarray*}
and this integral converges by (\ref{propG1}).
\end{proof}

Finally, the following elementary estimate is used in the proof of Lemma \ref{conditiondecay} and Lemma \ref{TheoremExponentialDecay}.
\begin{lemma}\label{Hopital}
For all $t>0$, $\theta\geq 0$, $\rho >0$, define
$$Z(t, \theta, \rho )=\int_0^t (s+1)^{-\theta}e^{\rho s}ds.$$
Then, for all $\theta >0$:
\begin{eqnarray}
\label{Hopital0}
&&Z(t, \theta, \rho )\leq [2^{\theta}(t+1)^{-\theta}+3e^{-\rho t/3}]\frac{e^{\rho  t}}{\rho }.
\end{eqnarray}
\end{lemma}
\begin{proof} 
We  define
$$S(t, \theta, \rho )=Z(t, \theta, \rho )\rho e^{-\rho  t}.$$
and split  $S(t)$ into two parts $S(t, \theta, \rho )=S_1(t, \theta, \rho )+S_2(t, \theta, \rho )$:
$$S_1(t, \theta, \rho )=\int_{t/2}^t \rho  (s+1)^{-\theta}e^{\rho (s-t)}ds,$$
$$S_2(t, \theta, \rho )=\int^{t/2}_0 \rho  (s+1)^{-\theta}e^{\rho (s-t)}ds.$$
In the first integral, we have
\begin{eqnarray*}
S_1(t, \theta, \rho )&=&\int_{t/2}^t \rho  (s+1)^{-\theta}e^{\rho (s-t)}ds\le \int_{t/2}^t \rho  (t/2+1)^{-\theta}e^{\rho  (s-t)}ds\\
&\leq& (t/2+1)^{-\theta}\int_{t/2}^t \rho  e^{\rho (s-t)}ds\le  (t/2+1)^{-\theta}(1- e^{- {\rho  t/2} })\\
&\leq& (t/2+1)^{-\theta}\le 2^{\theta}(t+1)^{-\theta}.
\end{eqnarray*}
For  the second integral we notice that, since $s\in (0, t/2)$ we have $s-t < -(s+t)/3$ and then
\begin{eqnarray*}
S_2(t, \theta, \rho )&=&\int_{0}^{t/2} \rho  (s+1)^{-\theta}e^{\rho (s-t)}ds\le \int_0^{t/2} \rho  (s+1)^{-\theta}e^{-\rho s/3}e^{-\rho t/3}ds\\
&\leq&e^{-\rho t/3}\int_0^{t/2} \rho  e^{-\rho s/3}ds\le 3e^{-\rho t/3}.
\end{eqnarray*}
\end{proof}
\subsection{The measure $\mathcal U$ and the function $M$.}
The following expressions for $\mathcal U(p, p')$ and $M(p)$  have been obtained in \cite{EPV}:
\bear
\mathcal U(p, p')=2\left|\mathcal M(p, p', p-p' )\right|^2\delta (\omega (p)-\omega (p')-\omega (p-p'))\times \nonumber\\
\times n_0(\omega (p))(1+n_0(\omega (p')))(1+n_0(\omega (p)-\omega (p')))+ \nonumber\\
+2\left|\mathcal M(p', p, p'-p )\right|^2\delta (\omega (p')-\omega (p)-\omega (p'-p))\times \nonumber\\
\times n_0(\omega (p'))(1+n_0(\omega (p)))(1+n_0(\omega (p')-\omega (p)))-\nonumber\\
-2\left|\mathcal M(p'+p, p, p') \right|^2\delta (\omega (p)+\omega (p')-\omega (p+p'))\times \nonumber\\
\times (1+n_0(\omega (p)))(1+n_0(\omega (p')))n_0(\omega (p)+\omega (p')). \label{S6EU}
\eear
\bear
M(p)=\frac {1} {\omega (p)}\int  _{ \RR^3 }\mathcal U(p, p')\omega (p')dp' \label{S6EM}
\eear

\subsubsection{The formal approximation argument.}
In the limit $|p|/4mgn_c\to 0$ we have:
\bean
&&\omega (p)=\left[\frac {g n_c} {m}|p|^2+\left(\frac {|p|^2} {2m} \right)^2 \right]^{1/2}= c(|p|+\psi (|p|))^{1/2}\\
&&0<\psi (|p|)=o(|p|^3).
\eean
In order to see how the equation (\ref{E51}) may be formally obtained from equation (\ref{E5i})--(\ref{E5iii}) we first  express the delta measures in $\mathcal U(p, p')$ in terms of 
$r=|p|$, $r'=|p'|$, and the angle $u=\cos \theta _{ p, p' }$.  We notice first that, given $r>0$ and $r'>0$, we call $u_1(r, r')$ the positive solution $u$ of the equation $\omega (p)-\omega (p')-\omega (p-p')=0$, or equivalently, of the equation:
$$
r=r'+(r^2+r'^2-2rr' u)^{1/2}
$$
We may then express:
\bean
\delta (\omega (p)-\omega (p')-\omega (p-p'))=
F_1(r, r')\delta(u- u_1 (r, r'))
\eean
with
\bean
F_1(r, r')&=&\frac {-1} {\frac {\partial h} {\partial u}(r, r', u_1 (r, r'))}\\
h(r, r', u)&=&\omega (p-p') 
\eean

An asymptotic expression may be obtained for $u_1 (r, r')$ in the limit $r\to 0$ and $r'\to 0$ as follows (cf. \cite{Benin}):
\bean
&&u_1 (r, r')=1 - \frac{r-r'}{rr'}(\psi (r)-\psi (r')-\psi (r-r'))+\\
&&\hskip 1cm +\mathcal O\left(\psi (r)^2+\psi (r')^2+2\psi (r)\psi (r') \right)\,\,\,\,\hbox{as}\,\, r\to 0, r'\to 0.
\eean
 
 Similar arguments yield:
\bean
\delta (\omega (p')-\omega (p)-\omega (p'-p))=
F_2(r, r')\delta(u- u_2(r, r'))\\
\delta (\omega (p+p')-\omega (p)-\omega (p'))=
F_3(r, r')\delta(u- u_3 (r, r'))
\eean
In the limit considered in this article we are approximating $\omega (p)$ as $c|p|$. Therefore,  the angles between the vectors $p$, $p'$ and $p-p'$ involved in  the collisions must all be equal to one.  This  corresponds to the approximation:
$$u_i(r, r')=1\,\,\,\,i=1, 2, 3.$$
The measure $\mathcal U(p, p')$ is then approximated as:
\bear
&&\mathcal U(p, p')\approx W(p, p')\equiv G(r, r') \delta (u-1)\label{S6EW}\\
&&G(r, r')=\frac {9c}{32\pi ^2mn_c}\left[|r r' (r-r')|^2F_1(r, r')\times \right. \nonumber\\
&&\hskip 1.7cm\times n_0(r)(1+n_0(r'))(1+n_0(r-r'))+ \nonumber\\
&&\hskip 1.7cm+|r r' (r-r')|^2F_2(r, r')\times\nonumber\\
&&\hskip 1.7cm\times n_0(r')(1+n_0(r))(1+n_0(r'-r))-\nonumber\\
&&\hskip 1.7cm-|r r' (r+r')|^2F_3(r, r')\times \nonumber\\
&&\hskip 1.7cm\left.\times (1+n_0(r))(1+n_0(r'))n_0(r+r')  _{\mathcal \null  } \right ]. \label{S6EG}
\eear

Using the rotational invariance of $W(p, p')$ we may write its expansion in terms of Legendre's polynomials:
\bean
 W(p, p')=\sum _{ \ell=0 }^\infty W_\ell (r, r') P _{ \ell }(\cos \theta (p, p'))
\eean
where $r=|p|$, $r'=|p'|$, $P_\ell $ is the Legendre polynomial of degree $\ell$ and $\theta (p, p')$ is the angle between $p$ and $p'$ and
\bean
W_\ell (r, r')=\frac {2\ell +1} {2}\int  _{ -1 }^1W P _{ \ell} (u)du ,\,\,\,u=\cos \theta (p, p').
\eean
It follows that, with some abuse of notation:
\bear
M(p)\equiv M(|p|)=M(r)=\frac {1} {\omega (cr)}\int _0^\infty W_0 (r, r') \omega (cr') r'^2dr', \label{SAEM0}
\eear
where we recall that $r=|p|$. On the other hand the function $W_0(r, r')$ is given by
\bear
W_0(r, r')=\frac {9(r-r')^2H(r-r')} {32\pi^2 n}n_0(r)(1+n_0(r'))(1+n_0(r-r'))+\nonumber\\
+\frac {9(r'-r)^2H(r'-r)} {32\pi^2 n}n_0(r')(1+n_0(r))(1+n_0(r'-r))-\nonumber\\
- \frac {9(r+r')^2} {32\pi ^2 n}n_0(r+r')(1+n_0(r))(1+n_0(r')). \label{SAEW0}
\eear
where,  we denote $n_0(r)=n_0(p)$,  and $H(r)$ is the Heaviside function (see \cite{EPV}).

\begin{prop}
\label{S6P1}
Let $M(p)$ be the function defined in (\ref{SAEM0}). Then,
\bear
\label{S6EMG}
M(p)\equiv M(r)=\Gamma (k)\,n_0(r)(1+n_0(r))
\eear
where $r=|p|$ and $k=cr/2k_BT$. Moreover:
\bear
\label{S6EMzero}
\lim _{ r\to 0 }\frac {M(r)\sinh^2 k} {k}=\frac {\pi ^4} {60}\\
\label{S6EMinf}
\lim _{ r\to \infty }\frac {M(r)\sinh^2 k} {k^5}=\frac {1} {60}.
\eear
\end{prop}
\begin{proof}
The function $f$  satisfies the equation (\ref{E8bis}).  Using that, by (\ref{E8}), $f(t, k)=(k/\sinh k)\Omega (t, p)$ and that $\Omega (t, p)$ satisfies equation  (\ref{E51}), identity (\ref{S6EMG}) follows. Properties (\ref{S6EMzero}) and (\ref{S6EMinf}) are then consequence of Lemma \ref{Gamma}.
\end{proof}
\subsection{Legendre's polynomials.} We recall that the Legendre's polynomial of degree $n\in \NN$ is defined as:
\bear
\label{S6ELegendre}
P_n(x)=\frac {1}2^n {n!}\frac {d^n} {dx^n}\left[\left(x^2-1 \right)^2\right], \,\,\, n=0, 1, 2, \cdots
\eear
These polynomials  form a complete orthogonal set of functions in $L^2(-1, 1)$ such that:
\bear
\label{S6Eorth}
\int  _{ -1 }^{1}P_m(x)P_n(x)dx=\frac {2} {2n+1}\delta  _{ m n }.
\eear
The following property of the Legendre's polynomials is useful to obtain fornula (\ref{S1E210}):
\bear
\label{S6LegHarm}
P _{ \ell }(u\cdot u')=\frac {4\pi } {2\ell+1}\sum _{ m=-\ell }^\ell Y _{ \ell\, m }(u)Y^* _{ \ell\, m }(u')
\eear
for all $u\in \Sx^2, u'\in \Sx^2$.\\

 {\bf Acknowledgements.} The work of M. E.  has been supported by  Grant 2011-29306-C02-00, MICINN, Spain and Basque Government Grant IT641-13. The second author has been supported by  Grant MTM2011-29306-C02-00, MICINN, Spain, ERC Advanced Grant FP7-246775 NUMERIWAVES, and Grant PI2010-04 of the Basque Government. Both authors are supported by    SEV-2013-0323, MINECO, Spain. T. M. B. would like to thank  his advisor, Professor Enrique Zuazua for fruitful  discussions on this work. Both authors acknowledge enlightening  conversations with Pr. M. A. Valle.
\bibliographystyle{plain}\bibliography{QuantumBoltzmannPhonons}

\end{document}